\documentclass[11pt,a4paper]{article}
\usepackage{preambule}
\usepackage{nccmath}

\newcommand{\Pbf}{\mathbf{P}}
\newcommand{\Ebf}{\mathbf{E}}
\newcommand{\TT}{\mathbb{T}}
\newcommand{\rme}{\mathrm{e}}

\title{Covering of an inner subset by the confined random walk}

\author{Nicolas Bouchot}

\AtEndDocument{\bigskip{\footnotesize%
		\textsc{Institute für Mathematik, Universität Innsbruck -- Technikerstraße 13, 7. OG
			A-6020 Innsbruck, Austria.} \par  
		\textit{E-mail address}: \texttt{nicolas.bouchot@uibk.ac.at} }}

\begin{document}
	
	\pagestyle{plain}
	
	\maketitle
	
	\begin{abstract}
		
		We consider the simple random walk conditioned to stay forever in a finite domain $D_N \subset \ZZ^d, d \geq 3$ of typical size $N$. This confined walk is a random walk on the conductances given by the first eigenvector of the Laplacian on $D_N$. On inner sets of $D_N$, the trace of this confined walk can be approximated by tilted random interlacements, which is a useful tool to understand some properties of the walk.
		
		In this paper, we propose to study the cover time of inner subsets $\Lambda_N$ of $D_N$ as well as the so-called late points of these subsets. If $\Lambda_N$ contains enough late points, we obtain the asymptotic expansion of the covering time as $c_\Lambda N^d \big[ \log N - \log\log N + \mathcal{G} \big]$, with $\mathcal{G}$ a Gumbel random variable, as well as a Poisson repartition of these late points. The method we use is similar to Belius' work about the simple random walk on the torus, which displays the same asymptotics albeit without the $\log \log N$ term. In the more general setting of ``ball-like'' $\Lambda_N$, we simply get the first term of the asymptotic expansion. 
		
		
	\end{abstract}	
	
	\section{Introduction}

	\paragraph{Motivations}
	
	This paper studies the covering of a finite domain $D_N \subset \ZZ^d, d \geq 3$ of typical size $N \gg 1$ by the simple random walk conditioned to stay in $D_N$ \emph{forever}, that is what we call the \emph{confined walk}. This process is in fact a random walk on the conductances given by the first discrete Laplace eigenfunction on $D_N$.
	
	Recently, the author proved in \cite{bouchot2024confinedrandomwalklocally} a connection between this confined walk and ``random interlacements'', a Poisson cloud of random walk trajectories, in the form of local couplings. We refer to the rest of the introduction for a rigorous definition.
	
	It is well-known that random interlacements can also be coupled to the simple random walk (SRW) on the torus (see \cite{windischRandomWalkDiscrete2008,teixeiraFragmentationTorusRandom2011,cernyRandomWalksTorus2016b}). Such coupling has been a powerful tool to study the behavior of the random walk, and still is to these days (see \cite{prévost2023phase} and \cite{chiariniLowerBoundsBulk2023} for two recent applications).
	
	Therefore, as an application of the work \cite{bouchot2024confinedrandomwalklocally}, it is natural to tackle the study of the confined walk through interlacements in the same way as for the SRW on the torus. The main difference between the SRW and the confined walk is the presence of a drift which results in inhomogeneities in the occupation measure of the confined walk.
	
	Our main result in Theorem \ref{th:asymptotics-cover-RW} gives the asymptotic covering time of subsets of $D_N$, as well as further asymptotics for some special subsets. We also give an application of interest in the case where $D_N$ is a ball and how it relates to the one-dimensional case.
	
	We stress that this work relies of the existing approaches for the covering of the torus by the SRW, which occurs around time $g(0) N^d \log N$ with $g(0)$ the Green function of SRW on $\ZZ^d$ at $0$. The main interest here is how the confined walk compares to the SRW and how these differences translates in the results and proofs.

	\subsection{Random walk confined in a large domain as a Doob transform}

	Fix a bounded connected open set $D \subset \RR^d, d \geq 3$ which contains the origin and has a smooth boundary (meaning it is given in local coordinates by a smooth function). Let $N \geq 1$ and define $D_N \defeq (ND) \cap \ZZ^d$ the discrete blow-up of $D$ with factor $N$.

	We consider the substochastic matrix $P_N$ of the simple random walk (SRW) on $\ZZ^d$ killed when exiting $D_N$, given by
	\[ P_N(i,j) = \frac{1}{2d} \quad \text{if $i,j \in D_N$ and $i \sim j$} \, , \quad 0 \quad \text{else} \, . \]
	Write $\lambda_N$ and $\phi_N$ for the first eigenvalue and associated eigenvector of $P_N$, with the following normalization:
	\begin{equation}
		\label{def:eigenfunction}
		P_N \phi_N = \lambda_N \phi_N \, , \quad \| \phi_N^2 \| \defeq \sum_{x \in D_N} \phi_N^2(x) = N^d \, .
	\end{equation}
	Note that $\phi_N$ is defined on $D_N$, but it might be convenient to extend it to $\partial D_N$ (and $\mathbb{Z}^d$) by setting $\phi_N(x) = 0$ for $x\notin D_N$.

	Let us introduce the following notation: for $h : \ZZ^d \longrightarrow \RR$  a real-valued function on $\ZZ^d$, we define
	\begin{equation}
		\label{def:barh-Deltad}
		\Delta_d h(z) \defeq \frac{1}{2d} \sum_{|e| = 1} h(z + e) - h(z) \, .
	\end{equation}
	This way, one can rewrite~\eqref{def:eigenfunction} as the Dirichlet problem $\Delta_d \phi_N = (1-\lambda_N) \phi_N$, with  boundary condition $\phi_N\equiv 0$ on $\partial D_N = \{ y \in \ZZ^d \setminus D_N \, : \, \exists x \in D_N, x \sim y \}$.We refer to \cite{berger2025propertiesprincipaldirichleteigenfunction} for more details and references.
	
	The \emph{confined walk} on $D_N$ is then defined as the Markov chain $(X_n)_{n \geq 0}$ with transition probabilities given by
	\begin{equation}\label{eq:def-p_N}
		p_N(x,y) \defeq \frac{\lambda_N^{-1}}{2d} \frac{\phi_N(y)}{\phi_N(x)} \indic{x \sim y} \quad , \quad \forall x,y \in D_N \, .
	\end{equation}
	We write $\Pbf^N_\mu$ for its law with starting distribution $\mu$; we will also write $\Pbf^N_x$ when $\mu = \delta_x$ and $\Pbf^N_{\phi_N^2}$ when $\mu = c_N \phi_N^2(\cdot)$ with the correct normalizing constant~$c_N$.
	Let us observe that the transition kernel from~\eqref{eq:def-p_N} is that of a random walk among conductances $c_N(x,y) \defeq \phi_N(x)\phi_N(y)$, therefore $\phi_N^2$ is an invariant measure of the confined walk.
	
	We give some properties of the confined walk and the eigenfunction $\phi_N$ in Section \ref{sec:details-confined-walk} below. For now, the most relevant of these properties is the uniform convergence of $\phi_N$ towards the solution to the following (continuous) Dirichlet problem:
	\begin{equation}
		\label{eq:Dirichlet}
		\begin{cases}
			\Delta  v = \mu\, v & \quad \text{ on } D \,,\\ 
			v=0 & \quad\text{ on } \partial D \,,
		\end{cases}
		\quad , \quad \| v \|_{L^2}^2 \defeq \int_D v(x)^2 \dd x = 1
	\end{equation}
	with $\Delta$ the usual Laplacian. This problem admits a sequence of solutions $(\mu_k,\varphi_k)_{k\geq 1}$ which is ordered by the eigenvalues $\lambda_1 \geq \lambda_2 \geq \dots$. We write $(\mu,\varphi) \defeq (\mu_1,\varphi_1)$ for simplicity. Note that $\varphi$ is $\mathscr{C}^\infty$ on the interior of $D$ (see \cite[Theorem 1.5]{berger2025propertiesprincipaldirichleteigenfunction}). We give a precise statement of the convergence $\phi_N(z \cdot N) \to \varphi(z)$ in Proposition \ref{prop:convergence-vp} below.

	\subsection{Covering of inner sets - main result}
	
	\subsubsection{General setting}

	From now on, we will focus on the covering time of an inner subset of $D_N$. We fix $\eps > 0$ and consider an open set $\Lambda \subset D$ such that $d(\Lambda,\partial D) \geq 2 \eps$. We define $\Lambda_N \defeq (N \Lambda) \cap \ZZ^d \subseteq D_N$ its discrete blowup of size $N$.
	
	An additional assumption that we make on $\Lambda_N$ is that it is $sN$-regular for some $s > 0$ in the sense of \cite[Def.~8.1]{popovSoftLocalTimes2015}: there is a $s > 0$ such that for any $N$ large enough, for any point $x \in \partial B_N$, there are balls $B^{\mathrm{in}} \subseteq \Lambda_N$, $B^{\mathrm{out}} \subseteq \ZZ^d \setminus \Lambda_N \cup \partial \Lambda_N$ of radius $s N$ that are both tangent to $\Lambda_N$ at point $x$.
	Note also that $\Lambda_N$ is macroscopic, and may be arbitrarily close to $D_N$, for example if $D$ is a ball, or if $D$ is well approximated by ``$s$-regular'' sets. This assumption is often refered to as a \emph{positive reach} assumption.
	
	We define the range of the confined walk $(X_n)_{n \geq 0}$ up to time $t \geq 0$ as the random set $\mathcal{R}_{\phi_N}(t) \defeq \{ X_0 , \dots , X_t \}$. The covering time of $\Lambda_N$ is then defined as
	\begin{equation}
		\mathfrak{C}_N(\Lambda_N) \defeq \inf \mathset{t \geq 0 \, : \, \Lambda_N \subseteq \mathcal{R}_{\phi_N}(t)} = \sup \big\{ H_x \, : \, x \in \Lambda_N \big\} \, ,
	\end{equation}
	with $H_x \defeq \inf \{ t \geq 0 \, : \, X_t = x \}$ the hitting time of $x$.
	Let us state the first of the main results of this paper, which gives asymptotics for $\mathfrak{C}_N(\Lambda_N)$. We recall that $g$ is the Green function of SRW on $\ZZ^d$.
	
	\begin{theorem}\label{th:asymptotics-cover-RW}
		Under $\Pbf^N_{\phi_N^2}$, we have the following asymptotics in probability:
		\begin{equation}\label{eq:cv-covering-Lambda}
			\mathfrak{C}_N(\Lambda_N) \sim g(0) \alpha_{\Lambda}^{-1} \lambda_N N^d \log |\Lambda_N| \quad \text{with} \quad \alpha_{\Lambda} \defeq \inf_{x \in \Lambda} \varphi^2(x) \, .
		\end{equation}
		Moreover, $\mathfrak{C}_N(\Lambda_N)$ has ``super-Gumbel'' fluctuation, in the sense that for any fixed $z \in \RR$,
		\begin{equation}\label{eq:super-gumbel}
			\liminf_{N \to +\infty} \Pbf^N_{\phi_N^2} \Big( \mathfrak{C}_N(\Lambda_N) \leq \mfrac{\lambda_N g(0)}{\alpha_{\Lambda}} N^d \{ \log |\Lambda_N| + z \} \Big) \geq \exp ( - e^{-z}) \, .
		\end{equation}
	\end{theorem}

	\subsubsection{Gumbel and Poisson behaviour under a stronger assumption}
	
	If we restrict ourselves to the case where $\alpha_\Lambda$ is achieved by a positive proportion of points in $\partial \Lambda$, we are able to get the fluctuations of the covering time as well as a description of the ``last points'' to be covered by the confined walk.
	Fix $\alpha \in \varphi^2(\Lambda)$, we define the $\alpha$-level set of $\varphi^2$ as $\mathcal{L}_{\alpha} \defeq \mathset{x \in D \, : \, \varphi^2(x) = \alpha}$.
	
	Our stronger assumption on $\Lambda$ is that $\partial \Lambda \cap \mathcal{L}_{\alpha_\Lambda}$ has in some sense a positive $(d-1)$ Lebesgue measure, meaning that it is the trace on $\partial \Lambda$ of an open set of $\RR^d$. We formulate this assumption with the following statement, which also means that $\alpha \mapsto \mathcal{L}_{\alpha}$ is somewhat continuous at $\alpha_\Lambda$.
	
	\begin{hypothesis}\label{hyp:deformation-level-sets}
		We assume that there exists a $\eps_0 > 0$ such that on $\big\{ x \in D \, : \, \alpha_\Lambda \leq \varphi^2(x) \leq (1+\eps_0) \alpha_\Lambda \big\}$, we have $\nabla \varphi^2 \neq 0$.
	\end{hypothesis}
	
	Note that Assumption \ref{hyp:deformation-level-sets} implies that $\eps^{-1} \big| \big\{ x \in D \, : \, \alpha_\Lambda \leq \varphi^2(x) \leq (1+\eps) \alpha_\Lambda \big\} \big|$ converges towards a positive limit (see Proposition \ref{prop:cardinal-level-sets} in the Appendix).
	As an illustration of Assumption \ref{hyp:deformation-level-sets}, one can consider the setting where $D$ and $\Lambda$ are concentric balls. In this case, the level sets are spheres and $\varphi = \hat{\varphi}(| \cdot |)$ decreases with the radius of the sphere. We give more details on this example in Section \ref{sec:application-boule}.
	
	Also note that under Assumption \ref{hyp:deformation-level-sets}, $\varphi^2$ achieves its minimum $\alpha_\Lambda$ on $\partial \Lambda$, meaning $\partial \Lambda \cap \mathcal{L}_{\alpha_\Lambda} \neq \varnothing$. In fact, we will see that the assumption forces the minimum to be achieved for a positive proportion of points of $\partial \Lambda$.

	For $z \in \RR$, we define
	\begin{equation}\label{eq-def:t_N^Lambda(z)}
		t_N^\Lambda(z) \defeq \mfrac{g(0)}{\alpha_\Lambda} N^d \big\{ \log |\Lambda_N| - \log\log |\Lambda_N| + z \big\} \, .
	\end{equation}
	We then have the following statement.
	
	\begin{theorem}\label{th:gumbel+poisson-RW}
		Assume that Assumption \ref{hyp:deformation-level-sets} holds and fix $z \in \RR$. We have
		\begin{equation}
			\lim_{N \to +\infty} \PP \big( \mathfrak{C}(\Lambda_N) \leq t_N^\Lambda(z) \big) = \exp \big( - \kappa_\Lambda e^{-z} \big) \quad \text{with} \quad \kappa_\Lambda \defeq \int_{\partial \Lambda \cap \mathcal{L}_{\alpha_\Lambda}} \frac{\dd x}{\big| \nabla \varphi^2 (x) \big|} \, .
		\end{equation}
		Moreover, there exists an explicit measure $\mu_\Lambda$, supported on $\partial \Lambda \cap \mathcal{L}_{\alpha_\Lambda}$ and with total mass $\kappa_\Lambda$, such that:
		\begin{equation}
			\mathcal{N}_{\mathrm{RW},N}^{\Lambda,z} \defeq \sum_{x \in \Lambda_N} \delta_{x/N} \indic{x \not\in \mathcal{R}_{\phi_N}(t_N^\Lambda(z))} = \sum_{x \in \Lambda_N} \delta_{x/N}  \indic{H_x > t_N^\Lambda(z)} \xrightarrow[N \to +\infty]{(d)} \mathcal{N}^{\Lambda,z} \, ,
		\end{equation}
		with $\mathcal{N}^{\Lambda,z}$ a Poisson point process on $\RR^d$ with intensity measure $e^{-z} \mu_\Lambda$. The convergence holds in distribution with respect to the weak topology on the space of point measures.
	\end{theorem}

	The proof of the two theorems heavily relies on the coupling of the range of the tilted RW with well-chosen random interlacements given by $\phi_N$. In the following sections, we properly introduce the \emph{tilted} random interlacements and state the coupling result that we use in this paper. The core of the paper will then be to prove Theorems \ref{th:cover-lambda-entrelac} \& \ref{th:gumbel+poisson-RI}, which are the analogues of Theorems \ref{th:asymptotics-cover-RW} \& \ref{th:gumbel+poisson-RW} for the tilted interlacements. We present how the later can be deduced from the former in Section \ref{sec:transfert-RI-RW}.
	
	\begin{remark}
		The $\log \log |\Lambda_N|$ correction in $t_N(0)$ is linked to the codimension of the set $\{ \varphi^2 = \alpha_\Lambda \}$. We conjecture that in the general setting, the next order asymptotics to Theorem \ref{th:asymptotics-cover-RW} is given by $n \log \log |\Lambda_N|$ where $n$ is chosen such that $\eps^{-n} \big| \{ x \in \Lambda : \varphi^2(x) \leq (1+\eps) \alpha_\Lambda \}\big|$ converges as $\eps \downarrow 0$ towards a finite positive limit. This accounts for the fact that very few points of $\Lambda_N$ are such that $\varphi^2 \approx \alpha_\Lambda$, therefore there are fewer ``last points'' to visit.
	\end{remark}

	
	
		

	\subsection{Comparison with the random walk on the torus}\label{sec:motivation-tore}
	
	Consider the discrete torus of size $N \geq 1$, denoted by $\TT_N^d \defeq (\ZZ/N\ZZ)^d$, as well as the simple, nearest-neighbour random walk (SRW) $S = (S_n)_{n \geq 0}$ on $\TT_N^d$ started from the uniform distribution. For $d \geq 3$, the behavior of this walk has been extensively studied in the literature, with notably the study of the last points of $\TT_N^d$ that the SRW visits. For $T \geq 0$, we define $\mathcal{R}_T(S) \defeq \big\{ S_0 , \dots , S_{\lfloor T \rfloor} \big\}$ the range up to time $T$. Then, the cover time $C_N^d$ of $\TT_N^d$ is defined as
	\begin{equation}\label{eqdef:cover-time-torus}
		C_N^d \defeq \inf \big\{t \in \NN \, : \, \mathcal{R}_t(S) = \TT_N^d \big\} \, .
	\end{equation}
	Write $\Pbf_x$ for the law of the simple random walk on $\ZZ^d$ with starting point $x \in \ZZ^d$, and simply $\Pbf = \Pbf_0$. With a slight abuse of notation, we also denote by $\Pbf_x/\Pbf$ the law of SRW on the torus.
	
	With the work of Aldous \cite{aldousTimeTakenRandom1983} on cover time of Markov chains, it is known that $C_N^d \sim g(0) N^d \log N$ in probability, where $g$ is the SRW Green's function. It is then natural to inquire about both the fluctuation of $C_N^d$ around $g(0) N^d \log N$ and about the points that are still not covered by the SRW at time $g(0) N^d \log N (1 \pm \bar{o}(1))$.
	
	Such questions where first tackled by Belius \cite{beliusGumbelFluctuationsCover2013}, who first proved Gumbel fluctuations for $C_N^d$:
	\begin{equation}\label{eq:gumbel-torus}
		\lim_{N \to +\infty} \sup_{x \in \TT_N^d} \Big| \Pbf_x\Big( C_N^d \leq g(0) N^d \big\{ \log N + z \big\} \Big) - \exp \big( - e^{-z} \big) \Big| = 0 \, .
	\end{equation}
	Regarding the ``late points'' of the SRW, Belius proved that these are distributed following a Poisson point process with uniform intensity measure. Given $z \in \RR$ we may define $\mathcal{N}_N^z$ the set of points that are not covered at time $g(0) N^d \big\{\log N + z \big\}$. Then, we have the following convergence in distribution:
	\begin{equation}
		\frac{1}{N} \mathcal{N}_N^z \xrightarrow[N \to +\infty]{(d)} \mathcal{N}^z \, ,
	\end{equation}
	where $\mathcal{N}^z$ is a Poisson point process on $[0,1]^d$ with intensity $e^{-z} d\ell_d$, with $\ell_d$ the Lebesgue measure on $\RR^d$.
	
	More recently, a work of Prévost, Rodriguez and Sousi \cite{prévost2023phase} improved our understanding of the late points and proved a sharp phase transition for the behavior of the set. Consider $a \in (0,1)$ and the $a$-fraction of the covering time $\tau_N^a \defeq a g(0) N^d \log N$. Then, if $\mathscr{L}_N^a$ is the set of points that are not covered by the random walk at time $\tau_N^a$, we have the following alternative:
	\begin{itemize}
		\item If $a > \tfrac12$, there is a coupling with a Bernoulli field $\mathscr{Z}_N^a$ such that $\mathscr{L}_N^a \approx \mathscr{Z}_N^a$ with probability going to $1$ as $N \to +\infty$\footnote{Here $\mathscr{L}_N^a \approx \mathscr{Z}_N^a$ means that $\mathscr{Z}_N^{a_-} \subseteq \mathscr{L}_N^a \subseteq \mathscr{Z}_N^{a_+}$ with $a_- = a_-(N) \sim a_+$.}. Morally, the presence of a late point can be considered as independent Bernoulli random variables.
		\item If $a = \tfrac12$, the best coupling possible has probability $e^{-1}$ as $N \to +\infty$ of having $\mathscr{L}_N^a \approx \mathscr{Z}_N^a$.
		\item If $a < \tfrac12$, there are no coupling that achieve a positive probability as $N \to +\infty$ for the event $\mathscr{L}_N^a \approx \mathscr{Z}_N^a$.
	\end{itemize}

	The key element in the proof of these results is a coupling that links the SRW to a random subset of $\ZZ^d$ introduced in \cite{sznitman2010vacant} and called the random interlacements. Informally speaking, the random interlacements (RI) at level $u > 0$, denoted by $\mathscr{I}(u)$, is a Poissonian collection of independent random walk trajectories whose ``density'' is governed by $u$. Successive works of \cite{windischRandomWalkDiscrete2008,teixeiraFragmentationTorusRandom2011,cernyRandomWalksTorus2016b} proved that for any $u > 0$, the range $\mathcal{R}_{u N^d}$ is ``locally close'' to $\mathscr{I}(u)$. More details about interlacements are provided in Section \ref{ssec:entrelacs}.
	
	With this coupling, we can link $C_N^d$ to $U_N^d$ the first level $u$ at which $\mathscr{I}(u)$ covers the torus $\TT_N^d$ by having $C_N^d \sim N^d \cdot U_N^d$. The main benefit of working with interlacements is that the cover level of a point has an explicit exponential distribution with parameter $g(0)^{-1}$ which corresponds to its \emph{capacity} (see Section \ref{ssec:entrelacs} below). Heuristically, we may think $U_N^d$ as a maximum of $N^d$ exponential variables, hence the Gumbel fluctuations which also propagate to $C_N^d$, therefore resulting in \eqref{eq:gumbel-torus}. The fact that Gumbel fluctuations appear in this context in well-known in extreme value theory when the exponential variables are i.i.d. (see \cite[Proposition 0.3]{resnickExtremeValuesRegular1987}). Random interlacements however display strong correlations: the main step of the proof of \eqref{eq:gumbel-torus} is then to create the required independence by studying the set of the last points to be covered.

	Our main motivation for this paper is the coupling between the confined walk and a random \emph{tilted} interlacements on macroscopic inner subsets of $D_N$ that was recently obtained in \cite{bouchot2024confinedrandomwalklocally}. As in the case of the torus, we will use this coupling to prove Theorems \ref{th:asymptotics-cover-RW} \& \ref{th:gumbel+poisson-RW} by first proving their analogs for the tilted interlacements.
	We stress that these analog results are the main point of this paper: the transfers to the confined walk, which we explain in Section \ref{sec:transfert-RI-RW} is mostly straightforward. The main additional difficulty comes in the fact that contrary to the case of the torus, the tilted interlacements is not spatially homogeneous. This can be translated in the previous heuristics by saying that the cover level of $x \in D_N$ is an exponential variable with position-dependent parameter (and not a constant like $g(0)^{-1}$).

	\paragraph{Acknowledgements} The author warmly thanks his PhD advisors Quentin Berger \& Julien Poisat for their continued support. This research was partially supported by the Austrian Science Fund (FWF) 10.55776/P34129.

	\section{Tilted interlacements, capacity and first estimates}

	\subsection{Tilting of random interlacements}\label{ssec:entrelacs}

	The tilted (continuous-time) random interlacements were introduced in \cite{liLowerBoundDisconnection2014} defined the tilting of continuous-time random interlacements as a way to locally modify the trajectories of the RI. 
	It was used in particular to get large deviation principles for disconnection events, by locally densifying the RI. 
	We choose to present here the point of view of \cite{teixeiraInterlacementPercolationTransient2009a} as random interlacements on weighted graphs, with a touch of local tilting of interlacements present in \cite{liLowerBoundDisconnection2014}. We also mention \cite{chiariniLowerBoundsBulk2023} for more recent example of their use in large deviation events.
	
	Consider the space of doubly infinite transient paths on the $d$-dimensional lattice $\ZZ^d, d \geq 3$:
	\[ W = \Big\{ w : \ZZ \rightarrow \ZZ^d \, : \, \forall n \in \ZZ, |w(n) - w(n+1)|_1 = 1 \text{ and } \lim_{|n| \to \infty} |w(n)| = \infty \Big\} \, , \]
	endowed with the $\sigma$-algebra $\mathscr{W}$ generated by the maps $w \mapsto w(n), n \in \ZZ$.
	We can define the equivalence relation $w \sim w' \Longleftrightarrow \exists k \in \ZZ, w(\cdot + k) = w'$, and we write $W^* \defeq W/\sim$ the corresponding quotient space as well as $\pi^*$ the associated canonical projection. The set $W^*$ is endowed with the $\sigma$-algebra $\mathscr{W}^*$ generated by $\pi^*$.
	
	Let $K$ be a subset of $\ZZ^d$. We define for $w \in W$ the hitting times
	\begin{equation}
		\label{def:HK}
		H_K(w) = \inf \mathset{k \in \ZZ \, : \, w(k) \in K}\,,
		\quad
		\bar{H}_K = \inf \mathset{k \geq 1 \, : \, w(k) \in K}
	\end{equation}
	with $\inf \varnothing = +\infty$ by convention.
	We also define $W_K = W \cap \mathset{H_K < +\infty}$ the set of trajectories that hit $K$.

	Consider a positive function $h : \ZZ^d \longrightarrow \RR_+^*$ which satisfies $h = 1$ outside a finite set. Informally, the $h$-tilted random interlacements of level $u > 0$, denoted by $\mathscr{I}_h(u)$, is a Poisson cloud of $h$-tilted random walk trajectories, \textit{i.e.}\ random walks on $\mathbb{Z}^d$ equipped with conductances $c(x,y) = h(x)h(y)$.
	
	More precisely, we let $\Pbf^{h}_z$ denote the law of the random walk on conductances $c(x,y)= h(x)h(y)$ starting at $z \in \ZZ^d$. Consider a finite set $K \subset \ZZ^d$. Then, we can define the $h$-tilted equilibrium measure of $K$ by
	\begin{equation}\label{eq:def:titled-equilibrium-mes}
		e^h_K(z) \defeq \mathbf{P}^h_z \big( \bar{H}_K = +\infty \big) h(z) \sum_{|e| = 1} h(z+e) \indic{z \in K}  \, , \quad \bar{e}^h_K(x) = \frac{e^h_K(x)}{\cpc^h(K)} \, .
	\end{equation}
	Here, $\cpc^h(K)$ is the \textit{tilted capacity} of the set $K$, given by
	\[ \cpc^h(K) = e^h_K(K) = \sum_{x \in \partial K} e^h_K(x) \, ,\]
	where we have used that the measure $e^h_K$ is supported on the (inner) boundary of $K$, which we denote by $\partial K = \{ x \in K \, : \, \exists y \in \ZZ^d \setminus K, x \sim y \}$.
	
	Intuitively, $\bar{e}^h_K$ is the law of the first entrance point in $K$ of a random walk trajectory that is coming from far away. Indeed, we have $\Pbf^h_z( X_{H_K} = x \, | \, H_K < +\infty) \to \bar{e}^h_K(x)$ as $|z| \to +\infty$, and the capacity can be alternatively written as
	\[
	\cpc^h(K) = \lim_{|z| \to +\infty} \frac{\Pbf^h_z (H_K < +\infty)}{G^h(z)} \, , \quad \text{with $G^h(z) = \sum_{n \geq 0} \Pbf^h_z(X_n = z)$ the Green function} \, ,
	\] 
	and can be interpreted as the \og{}size\fg{} of $K$ seen from a random walk starting at a faraway point on $\ZZ^d$. Since $h \equiv 1$ outside a finite set, the tilted Green function $G^h$ resembles the Green function of the SRW, hence we can convince ourselves that as $|z| \to +\infty$, we have $G^h(z) \sim G(z) \asymp |z|^{2-d}$, where $|\cdot|$ is the Euclidean norm on~$\RR^d$ (see \cite[Thm.~4.3.1]{lawlerRandomWalkModern2010}).
	
	Now, let $w^* \in W^*_K$ be a class of paths hitting $K$, and denote by $\tilde{w}$  the unique $w \in w^*$ such that $H_K(w) = 0$. For $K$ finite, we can define on $\pi^*(W_K)$ a finite measure $\nu_K$ given by
	\begin{equation}\label{eq:mesures-nu_K}
		\nu^h_K(w^*) = \Pbf^h_{\tilde{w}(0)}(w(\ZZ_-) \, | \, \bar{H}_K = +\infty) \times e^h_K(w(0)) \times \Pbf^h_{\tilde{w}(0)}(w(\ZZ_+)) \, ,
	\end{equation}
	which can be interpreted, after normalization, as the law of a SRW trajectory that hits~$K$. The measures $\nu^h_K$ for $K$ finite can be extended to a $\sigma$-finite measure $\nu^h$ on $W^*$.
	
	The tilted random interlacements process is a Poisson Point Process $\chi^h$ on the space $W^* \times \RR_+$ with intensity $\nu^h \otimes du$. The random interlacements of level $u > 0$ is the subset $\mathscr{I}_h(u) \subset \ZZ^d$ defined as
	\begin{equation*}
		\mathscr{I}_h(u) \defeq \mathset{z \in \ZZ^d \, : \, \exists (w,v) \in \chi^h, v \leq u, z \in w(\ZZ)} =  \mathset{w(k) \, : \, k \in \ZZ, (w,v) \in \chi^h, v \leq u} \, .
	\end{equation*}
	We denote by $\PP_h$ the law of the RI on the space $\mathcal{M}_p(W^* \times \RR_+)$ of point measures on $W^* \times \RR_+$, which we also use to denote the law of the (non-decreasing) family of random subsets $(\mathscr{I}_h(u))_{u > 0}$.
	
	A key property of RI is that its trace in a finite set $K$ can be recovered from a collection of tilted RW trajectories --- a property shared with the RI, see \cite{sznitman2010vacant}. More precisely, denote by $N^{h,u}_K$ a Poisson variable with parameter $u \cpc^h(K)$ and let $(X^{(j)})_{j \geq 0}$ be i.i.d.\ $h$-tilted RW trajectories starting from $\bar{e}^h_K$. We then have
	\begin{equation}\label{eq:entrelac-tilt-simul}
		\mathscr{I}_h(u) \cap K \overset{(d)}{=} \big\{ \mathcal{R}_\infty(X^{(j)}) \, : \, 1 \leq j \leq N^{h,u}_K \big\} \cap K \, .
	\end{equation}
	In particular, we deduce the following crucial identity, which characterizes the law of $\mathscr{I}_h(u)$ as a random subset of $\ZZ^d$ (see \cite[Remark 2.2-(2)]{sznitman2010vacant}): for any fixed $u > 0$ and any finite set $K \subset \ZZ^d$,
	\begin{equation}\label{eq:proba-vacant-tilt}
		\PP_h \big( \mathscr{I}(u) \cap K = \varnothing \big) = \exp \big(-u \cpc^h(K) \big) \, .
	\end{equation}

	\subsection{Cover level of the tilted interlacements}
	
	Recall that we consider a subset $\Lambda_N = (N \cdot \Lambda) \cap \ZZ^d \subset D_N$ with $\Lambda$ an open subset of $D$ that satisfies $d(\Lambda,\partial D) \geq 2 \eps$ for some $\eps > 0$.
	Let us define $\Lambda_N^{\eps} = \big\{ x \in D_N \, : \, d(x,D_N) \leq \eps N \big\} \subset D_N$ the $\eps N$-enlarged version of $\Lambda_N$. The tilting functions that we will consider are given by
	\begin{equation*}
		\Psi_N(x) \defeq \begin{dcases}
			\phi_N(x) & \quad \text{ if }\ x \in \Lambda_N^\eps \,, \\
			1 & \quad \text{ else} \, .
		\end{dcases} 
	\end{equation*}
	Write $\Pbf^{\Psi_N}_x$ for the law of the random walk on conductances $\Psi_N(i)\Psi_N(j)$ starting at $x \in \ZZ^d$, \textit{i.e.}\ with transition kernel denoted by $p_{\psi_N}(x,y)$.
	
	As previously mentioned, the main ordeal of the paper is to prove an analog of Theorem \ref{th:asymptotics-cover-RW} for the tilted interlacements $\mathscr{I}_{\Psi_N}$. To this end, we define the covering \emph{level} of the set $\Lambda_N$ as the random variable
	\begin{equation}
		\mathfrak{U}_N(\Lambda_N) \defeq \inf \{ u > 0 \, : \, \Lambda_N \subset \mathscr{I}_{\Psi_N}(u) \}  = \sup \big\{ U_x \, : \, x \in \Lambda_N \big\} \, ,
	\end{equation}
	with $U_x \defeq \inf \{ u > 0 \, : \, x \in \mathscr{I}_{\Psi_N}(u) \}$ the covering level of $x$.
	
	\begin{theorem}[Cover level of the interlacements]\label{th:cover-lambda-entrelac}
		Under $\PP_{\Psi_N}$, we have the following asymptotics in probability:
		\begin{equation}\label{eq:cover-level-Lambda-asymptotics}
			\mathfrak{U}_N(\Lambda_N) \sim g(0) \alpha_\Lambda^{-1} \log |\Lambda_N| \quad \text{with} \quad \alpha_\Lambda = \inf_{x \in \Lambda} \varphi^2(x) \, .
		\end{equation}
		Moreover, $\mathfrak{U}_N(\Lambda_N)$ has ``super-Gumbel'' fluctuation, in the sense that for any fixed $z \in \RR$,
		\begin{equation}\label{eq:super-gumbel-general-RI}
			\liminf_{N \to +\infty} \PP_{\Psi_N} \Big( \mathfrak{U}_N(\Lambda_N) \leq \mfrac{g(0)}{\alpha(\Lambda)} \{ \log |\Lambda_N| + z \} \Big) \geq \exp (-e^{-z}) \, .
		\end{equation}
	\end{theorem}

	Let us define
	\begin{equation}\label{eq:def-u_N(z)}
		u_N^{\Lambda}(z) \defeq \frac{t_N^\Lambda(z)}{N^d} = \frac{g(0)}{\alpha_\Lambda} \big\{ \log |\Lambda_N| - \log\log |\Lambda_N| + z \big\} \, .
	\end{equation}
	Our second theorem states that provided Assumption \ref{hyp:deformation-level-sets} holds, $\mathfrak{U}_N^{\Lambda} \approx u_N^{\Lambda}(\mathcal{G})$ with $\mathcal{G}$ a Gumbel random variable.
	
	\begin{theorem}\label{th:gumbel+poisson-RI}
		Assume that Assumption \ref{hyp:deformation-level-sets} holds and fix $z \in \RR$. We have
		\begin{equation}\label{eq-th:gumbel-RI}
			\lim_{N \to +\infty} \PP_{\Psi_N} \big( \mathfrak{U}(\Lambda_N) \leq u_N^\Lambda(z) \big) = \exp \big( - \kappa_\Lambda e^{-z} \big) \, ,
		\end{equation}
		where $\kappa_\Lambda$ is the constant in Theorem \ref{th:gumbel+poisson-RW}. Moreover, the following convergence holds in distribution with respect to the weak topology on the space of point measures:
		\begin{equation}\label{eq-th:poisson-RI}
			\mathcal{N}_{\mathrm{RI},N}^{\Lambda,z} \defeq \sum_{x \in \Lambda_N} \delta_{x/N} \indic{x \not\in \mathscr{I}_{\psi_N}(u_N^\Lambda(z))} = \sum_{x \in \Lambda_N} \delta_{x/N}  \indic{U_x > u_N^\Lambda(z)} \xrightarrow[N \to +\infty]{(d)} \mathcal{N}^{\Lambda,z} \, ,
		\end{equation}
		with $\mathcal{N}^{\Lambda,z}$ the same Poisson point process as the one in Theorem \ref{th:gumbel+poisson-RW}.
	\end{theorem}

	As explained in Section \ref{sec:motivation-tore}, we can link Theorem \ref{th:cover-lambda-entrelac} to the capacity of a point. Indeed, \eqref{eq:proba-vacant-tilt} yields
	\begin{equation}\label{eq:proba-point-vacant}
		\forall x \in \ZZ^d, \forall u > 0 \, , \quad \PP_{\Psi_N}(x \not\in \mathscr{I}_{\Psi_N}(u)) = \exp \big( - u \cpc^{\Psi_N}(\{x\}) \big) \, .
	\end{equation}
	
	Note that contrary to the case of simple random walk, $\cpc^{\Psi_N}(\{x\})$ is not a constant. Therefore, our first task is to understand its dependence in $x$. The following proposition is crucial, and explains why $\varphi^2$ appears in Theorem \ref{th:cover-lambda-entrelac}.
	
	\begin{proposition}\label{prop:cap-tilt=phi2}
		There are constants $c, C > 0$ that only depend on the dimension $d \geq 3$ such that for all $N$ large enough,
		\begin{equation}
			\sup_{x \in \Lambda_N} \bigg| \frac{\cpc^{\Psi_N}(\{ x \})}{\varphi_N^2(x)/g(0)} - 1 \bigg| \leq C N^{-c} \, .
		\end{equation}
	\end{proposition}
	
	We prove Proposition \ref{prop:cap-tilt=phi2} in the next section. For now, we state a useful corollary that will be used extensively in the rest of the paper in order to replace $\cpc^{\Psi_N}(\{x\})$ with $\varphi_N^2(x)$.
	
	\begin{corollary}\label{cor:sum-cap-sum-varphi}
		Let $f: \Lambda_N \longrightarrow \RR_+$ be non-zero and $u_N$ be such that $u_N N^{-c} \to 0$ (where $c$ is the constant in Proposition \ref{prop:cap-tilt=phi2}). Then, we have
		\begin{equation}
			\sum_{x \in \Lambda_N} f(x) \PP \big( x \not\in \mathscr{I}_{\Psi_N}(u_N) \big) \bigg/ \sum_{x \in \Lambda_N} f(x) \exp \Big( - u_N \varphi_N^2(x)/g(0) \Big) \xrightarrow[N \to +\infty]{} 1 \, .
		\end{equation}
	\end{corollary}
	
	\begin{remark}
		Proposition \ref{prop:cap-tilt=phi2} explains why the $\log \log N$ term appears in \eqref{eq-def:t_N^Lambda(z)} and \eqref{eq:def-u_N(z)} while being absent in \eqref{eq:gumbel-torus}. Indeed, when combined with the smoothness of $\varphi_N$, the Proposition implies that for $x,y \in \Lambda_N$,
		\begin{equation}\label{eq:diff-capacite-distance}
			\cpc^{\Psi_N} (\{ x \}) = \cpc^{\Psi_N}(\{ y \}) + g(0) \nabla \varphi^2 (x/N) \frac{|x-y|}{N} (1 + \bar{o}(1)) \, .
		\end{equation}
		Recall that on the torus, we can explain the $\log N$ factor by the extreme value theory: it corresponds to the logarithm of the number of ``relevant'' points, which are the points $z \in \TT_N^d$ such that $\PP(U_z \geq U_N^d) \asymp \sup_{w \in \Lambda_N} \PP(U_w \geq U_N^d)$. In the case of the torus, as $\cpc(\{w\})$ is constant, so is $\PP(U_w \geq U_N^d)$ therefore the set of relevant points is $\Lambda_N$ (hence a factor $\log |\Lambda_N|$).
		In our case however, by \eqref{eq:proba-point-vacant} and Proposition \ref{prop:cap-tilt=phi2}, one can easily show that
		\begin{equation*}
			\begin{split}
				\PP\Big(U_z \geq \tfrac{g(0)}{\alpha_\Lambda} \log |\Lambda_N| \Big) \asymp \sup_{w \in \Lambda_N} \PP\Big(U_w \geq \tfrac{g(0)}{\alpha_\Lambda} \log |\Lambda_N| \Big) \Longleftrightarrow 
				\cpc^{\Psi_N} (\{ z \}) - \frac{\alpha_\Lambda}{g(0)} = \grdO\Big( \frac{1}{\log |\Lambda_N|} \Big) \, .
			\end{split}
		\end{equation*}
		Using \eqref{eq:diff-capacite-distance}, this condition is satisfied only for $\asymp \frac{|\Lambda_N|}{\log |\Lambda_N|}$ points, which are the relevant points in this case. In particular, when considering the extreme value theory, we are left with a factor $\log (|\Lambda_N|/\log |\Lambda_N|) = \log |\Lambda_N| - \log \log |\Lambda_N|$, thus explaining the $\log \log |\Lambda_N|$ correction in \eqref{eq:def-u_N(z)}.
	\end{remark}
	

	\section{Some useful estimates}

	\subsection{Useful facts about the confined walk and eigenfunction}\label{sec:details-confined-walk}
	
	\subsubsection{Probabilistic interpretation as a confined walk}
	
	Recall that $\Pbf^N$ is the law of the confined walk and $\Pbf$ is the law of SRW on $\ZZ^d$.
	It is known, by standard Markov chain theory (see e.g. \cite[Appendix A.4.1]{lawlerRandomWalkModern2010}), that the transition kernel~\eqref{eq:def-p_N} is in some sense the limit, as $T \to +\infty$, of the transition kernels of the SRW conditioned to stay in $D_N$ until time $T$. Indeed, if $\tau_N$ is the first time the SRW exits $D_N$ and $x \sim y$, we have
	\begin{equation*}
		\frac{\Pbf_x(S_1 = y, \tau_N > T)}{\Pbf_x(\tau_N > T)} = \frac{1}{2d} \frac{\Pbf_y(\tau_N > T-1)}{\Pbf_x(\tau_N > T)}  = \frac{\lambda_N^{-1}}{2d} \frac{\lambda_N \Pbf_y(\tau_N > T-1)}{\Pbf_x(\tau_N > T)} \xrightarrow[T \to +\infty]{} \frac{\lambda_N^{-1}}{2d} \frac{\phi_N(y)}{\phi_N(x)} \, ,
	\end{equation*}
	the last limit being a consequence of the fact that $\Pbf_z(\tau_N > T) = \phi_N(z) \frac{\| \phi_N \|_1}{\| \phi_N \|_2^2} (\lambda_N)^T + \grdO((\beta_N)^T)$ as $T\to\infty$, for some $\beta_N < \lambda_N$ (see \cite[Prop.~6.9.1]{lawlerRandomWalkModern2010}).
	
	We also have a useful relation to compare the simple and tilted random walks.
	Consider a set $\Lambda$ which intersects $D_N$, and an event $A \in \mathcal{F}_{H_{\Lambda}}$, \textit{i.e.}\ an event that depends on the trajectory of the random walk until it hits $\Lambda$.
	Then, using the transition kernel $\tilde p_N(x,y)$ from~\eqref{eq:def-p_N} and after telescoping the ratios of the $\phi_N$'s, we have
	\begin{equation}
		\label{eq:PtildeP}
		\Pbf_x^N(A) = \frac{1}{\phi_N(x)} \Ebf_x\Bigg[ \mathbbm{1}_A \cdot (\lambda_N)^{-H_{\Lambda}} \phi_N(S_{H_{\Lambda}}) \indic{H_{\Lambda} < H_{\ZZ^d \setminus D_N}} \Bigg] \,.
	\end{equation}
	We can think of \eqref{eq:PtildeP} as a Feynman-Kac representation of the eigenvector $\phi_N$.
	
	\subsubsection{Properties of the eigenvector}
	
	Studying the confined walk requires some understanding of the eigenvector $\phi_N$. With Q. Berger, we investigated in \cite{berger2025propertiesprincipaldirichleteigenfunction} some properties of $\phi_N$ as $N \to +\infty$, which we will use in this paper. Let us compile these results to facilitate their use. Recall that we are in the setting of positive reach.
	
	\begin{proposition}[Regularity]\label{prop:reg-Phi_N}
		There is a constant $C>0$ (that depends only on the domain~$D$) such that, for any $x,y\in D_N$
		\begin{equation}\label{eq:reg-sup}
			\big| \phi_N(y)-\phi_N(x) \big| \leq C \, \frac{d(x,y)}{N}\, .
		\end{equation}
	\end{proposition}
	
	Note that this regularity implies that ratios of $\phi_N$'s in the bulk of $D_N$ are $0$ and $+\infty$. According to \cite[Corollary 1.14]{berger2025propertiesprincipaldirichleteigenfunction}, there is a positive constant $\kappa_1$, and some $N_1 \geq 1$ such that, for all $N\geq N_1$,
	\begin{equation}\label{eq:encadrement-ratio}
		\kappa_1 \leq \inf_{x,y \in B_N^\eps} \frac{\phi_N(x)}{\phi_N(y)} \leq \sup_{x,y \in B_N^\eps} \frac{\phi_N(x)}{\phi_N(y)} \leq \frac{1}{\kappa_1} \, .
	\end{equation}	
	This control of the ratios, combined with bounds on $\phi_N$, extends to a control on the values of $\phi_N$. More precisely, there is a positive constant $\kappa_2$, and some $N_2 \geq 1$ such that, for all $N\geq N_2$,
	\begin{equation}\label{eq:phiN-borne}
		\kappa_2 \leq \inf_{x \in B_N^\eps} \phi_N(x) \leq \sup_{x \in B_N^\eps} \phi_N(x) \leq \frac{1}{\kappa_2} \, .
	\end{equation}

	\begin{proposition}[Convergence]\label{prop:convergence-vp}
		For any $\eta > 0$, define $D_N^{\eta}\defeq \{x\in D_N, d(x,\partial D_N)  > \eta N\}$.
		Then there exists a positive constant $c_\eta$ such that
		\begin{equation}
			\sup_{x\in D_N^{\eta}} \bigg| \frac{\phi_N(x)}{\varphi(x/N)} - 1 \bigg| \leq c_\eta N^{-1} \,.
		\end{equation}
	\end{proposition}

	\subsubsection{On the principal eigenvalue}
	
	Let us mention that it is known (see \cite{pjm/1103040107} or \cite{berger2025propertiesprincipaldirichleteigenfunction} for details) that $\lambda_N$ satisfies
	\begin{equation}\label{eq:encadrement-lambda}
		\lambda_N = 1 - \frac{\lambda}{2d} \frac{1}{N^2} (1+\bar{o}(1)) \, , 
	\end{equation}
	where $\lambda=\lambda_D$ is the first Dirichlet eigenvalue of the Laplace-Beltrami operator on $D$ and $\bar{o}(1)$ is a quantity that vanishes as $N \to +\infty$.
	
	In particular, there exists $c_0 > 0$ a universal constant, which can be made arbitrarily close to $\lambda$ by taking $N$ large enough, such that for any $T \geq 0$, 
	\begin{equation}
		\label{eq:encadrement-lambda2}
		1 \leq \lambda_N^{-T} \leq \rme^{c_0 T/N^2}\,.
	\end{equation}
	Therefore, combining \eqref{eq:PtildeP} and \eqref{eq:encadrement-ratio} yields the following inequalities:
	\begin{equation}\label{eq:encadrement-proba}
		\kappa_1 \Pbf_x(A) \leq \Pbf^N_x (A) \leq \frac{1}{\kappa_1} \left[ \rme^{c_0} \Pbf_x(A) + \sum_{k = 1}^{+\infty} \rme^{c_0 (k+1)} \Pbf_x(A,\tau_C \in [k,k+1) N^2) \right]  \, .
	\end{equation}

	\subsection{Green function of the tilted walk, proof of Proposition \ref{prop:cap-tilt=phi2}}

	We introduce the Green function of the tilted RW:	
	\[
	G^{\Psi_N}(x,y)= \sum_{n \geq 0}  \Pbf^{\Psi_N}_x(X_n = y) \,.
	\]
	Let us stress that this Green function is \underline{not} symmetric in $x$ and $y$. When $x=y$, we simply write $G^{\Psi_N}(x) \defeq G^{\Psi_N}(x,x)$ which, contrary to the Green function of the SRW, does depend on $x$.
	
	\begin{proposition}[Last exit decomposition]\label{prop:LED-green-tilt}
		For any $K \subseteq B$ and any $x \in \ZZ^d$,
		\begin{equation}
			\Pbf^{\Psi_N}_x(H_K < +\infty) = \sum_{z \in K} G^{\Psi_N}(x,z) \frac{e_K^{\Psi_N}(z)}{\lambda_N \phi_N^2(z)} \, .
		\end{equation}
	\end{proposition}
	
	\begin{proof}
		Consider $x \in \ZZ^d$ and define $L_K \defeq \sup\{n\geq 0, X_n\in K\}$ the last time the tilted RW is in $K$. Since the $\Psi_N$-tilted RW is transient, we have $L_K < +\infty$ $\Pbf^{\Psi_N}_x$-a.s. 
		Therefore,
		\[ \begin{split} 
			\Pbf^{\Psi_N}_x(H_K < +\infty) = \sum_{n \geq 0} \sum_{y \in K} \Pbf^{\Psi_N}_x(L_K = n, X_n = y) &= \sum_{n \geq 0} \sum_{y \in K} \Pbf^{\Psi_N}_x(X_n = y) \Pbf^{\Psi_N}_y(\bar{H}_K = +\infty) \\
			&= \sum_{y \in K} G^{\Psi_N}(x,y) \Pbf^{\Psi_N}_y(\bar{H}_K = +\infty) \, .
		\end{split} \]
		By \eqref{eq:def:titled-equilibrium-mes}, we have $\Pbf^{\Psi_N}_y(\bar{H}_K = +\infty) = e_K^{\Psi_N}(y) / \lambda_N \phi_N^2(y)$ for all $y \in K \subseteq B$, which proves the lemma.
	\end{proof}

	\begin{proposition}\label{prop:encadrement-green-tilté}
		There are positive constants $c_d, C_d$, independent of $N$ such that, for all $x,y \in \Lambda_N$ with $|x-y|$ large enough,
		\begin{equation}
			\frac{c_d}{|x-y|^{d-2}} \leq G^{\Psi_N}(x,y) \leq \frac{C_d}{|x-y|^{d-2}} \, .
		\end{equation}
	\end{proposition}
	
	The proof we present in Appendix \ref{app:green-function} is inspired by the heuristics given in \cite[Remark 2.10]{drewitz2014introduction}. A more general proof should be possible using Gaussian bounds (see \cite{delmotte1999parabolic}). With such estimate on the Green function, we are now in the position to prove Proposition \ref{prop:cap-tilt=phi2}.

	\begin{proof}[Proof of Proposition \ref{prop:cap-tilt=phi2}]
		Fix $\gamma \in (0,1)$ and consider $B_N^\gamma$ the ball centered at $x$ with radius $N^\gamma$. We write the tilted capacity from $x$ to $\partial B_N^\gamma$ as
		\begin{equation}
			\mathcal{C}_{\Psi_N}(x \rightarrow \partial B_N^\gamma) \defeq \Pbf^{\Psi_N}_x (\bar{H}_x > H_{\partial B_N^\gamma}) = \inf \mathset{\mathcal{E}_{\Psi_N} (f) \, : \, f(x) = 1 \text{ and } \forall z \in \partial B_N^\gamma, f(z) = 0} \, ,
		\end{equation}
		where $\mathcal{E}_{\Psi_N} (f) \defeq \sum_{z \sim w} [f(z)-f(w)]^2 c_N(z,w)$ is the Dirichlet energy associated to the conductances $c_N(z,w) = \Psi_N(z) \Psi_N(w)$ (for $z \sim w$) (see \cite[Exercice 9.9]{levin2017markov}). We also write $\mathcal{C}_1(x \rightarrow \partial B_N^\gamma)$ for the capacity associated to the SRW (with $c_N(z,w) = 1$).
		We will prove the following chain of approximation holds with polynomially decreasing error:
		\[ \cpc^{\Psi_N}(\{ x \}) \approx \mathcal{C}_{\Psi_N}(x \rightarrow \partial B_N^\gamma) \approx \phi_N^2(x) \mathcal{C}_1(x \rightarrow \partial B_N^\gamma) \approx \phi_N^2(x) g(0)^{-1} \, . \]
		
		First note that we have $\Pbf^{\Psi_N}_z(H_x < +\infty) = G^{\Psi_N}(z,x)/G^{\Psi_N}(x)$ and Proposition \ref{prop:encadrement-green-tilté} combined with $G^{\Psi_N}(x) \geq 1$ imply that this probability is bounded from above by $C_d N^{\gamma(2-d)}$ uniformly in $z \in \partial B_N^\gamma$. In particular,
		\begin{equation}
			\begin{split}
				\Pbf^{\Psi_N}_x (\bar{H}_x > H_{\partial B_N^\gamma}) - \Pbf^{\Psi_N}_x (\bar{H}_x = +\infty) &= \sum_{z \in \partial B_N^\gamma} \Pbf^{\Psi_N}_x (\bar{H}_x < H_{\partial B_N^\gamma}, X_{H_{\partial B_N^\gamma}} = z) \Pbf^{\Psi_N}_z(H_x < +\infty) \\
				&\leq \Pbf^{\Psi_N}_x (\bar{H}_x < H_{\partial B_N^\gamma}) C_d N^{\gamma(2-d)} \, ,
			\end{split}
		\end{equation}
		hence proving that
		\begin{equation}\label{eq:cap-infty-cap-boule}
			1 - \frac{\cpc^{\Psi_N}(\{ x \})}{\mathcal{C}_{\Psi_N}(x \rightarrow \partial B_N^\gamma)} \leq C_d N^{\gamma(2-d)} \, .
		\end{equation}
		Let us focus on $\mathcal{C}_{\Psi_N}(x \rightarrow \partial B_N^\gamma)$. Using Proposition \ref{prop:reg-Phi_N} and the fact that $\phi_N \in [\kappa, \tfrac{1}{\kappa}]$ uniformly in $N$, there is a constant $c > 0$ such that for every function $f$,
		\begin{equation}\label{eq:comparaison-energie-dirichlet}
			\Big| \mathcal{E}_{\Psi_N} (f) - \phi_N^2(x) \mathcal{E}_{1} (f) \Big| \leq \frac12 \sum_{z \sim w} [f(z)-f(w)]^2 \big| \phi_N(z)\phi_N(w) - \phi_N^2(x) \big| \leq c N^{\gamma-1} \phi_N^2(x) \mathcal{E}_{1} (f) \, .
		\end{equation}
		Since $N^{\gamma-1} \phi_N^2(x) \leq \kappa^{-2} N^{\gamma-1} \to 0$, combining \eqref{eq:cap-infty-cap-boule} and \eqref{eq:comparaison-energie-dirichlet}, we easily deduce that
		\begin{equation}
			\Big| \frac{\mathcal{C}_{\Psi_N}(x \rightarrow \partial B_N^\gamma)}{\phi_N^2(x) \mathcal{C}_{1}(x \rightarrow \partial B_N^\gamma)} - 1 \Big| \leq c_2 N^{\gamma-1} \, .
		\end{equation}
		To conclude the proof, we claim that $0 \leq \mathcal{C}_{1}(x \leftrightarrow \partial B_N^\gamma) g(0) - 1 \leq c_3 N^{\gamma(2-d)}$. This can be proved with the same method as previously by using $g(x,y) \leq C_d' |x-y|^{2-d}$ (see \cite[Theorem 4.3.1]{lawlerRandomWalkModern2010}). Finally, using Proposition \ref{prop:convergence-vp}, we have $\|\phi_N^2 - \varphi_N^2 \|_\infty \leq N^{-\mu}$ for $\mu = 1/2(d+1)$. Therefore, taking $c = \mu \wedge (1-\gamma) \wedge \gamma(d-2)$ with the optimal $\gamma = 1/(d-1)$, the proposition follows.
	\end{proof}

	\subsection{Estimates on the capacity of certain sets}

	\begin{proposition}\label{prop:cap-2-points-exacte}
		Let $x,y \in B_N$, then
		\begin{equation}
			\cpc^{\Psi_N}(\mathset{x,y}) = \frac{\cpc^{\Psi_N}(\{ x \}) \Pbf^{\Psi_N}_x(H_y = +\infty) + \cpc^{\Psi_N}(\{ y \}) \Pbf^{\Psi_N}_y(H_x = +\infty)}{1 - \Pbf^{\Psi_N}_x(H_y < +\infty)\Pbf^{\Psi_N}_y(H_x < +\infty)} \, .
		\end{equation}
	\end{proposition}
	
	\begin{proof}
		Write $K = \mathset{x,y}$ and use the last exit decomposition at points $x,y \in K$. we have
		\begin{equation*}
			1 = G^{\Psi_N}(x) \frac{e_K(x)}{\lambda_N \phi_N(x)^2} + G^{\Psi_N}(x,y) \frac{e_K(y)}{\lambda_N \phi_N(y)^2} = G^{\Psi_N}(y,x) \frac{e_K(x)}{\lambda_N \phi_N(x)^2} + G^{\Psi_N}(y) \frac{e_K(y)}{\lambda_N \phi_N(y)^2} \, .
		\end{equation*}
		Write $f(z) = e_K(z) / \lambda_N \phi_N(z)^2$, then we can solve this system for $f(x),f(y)$:
		\[ f(x) = \frac{1}{G^{\Psi_N}(x)} \left[ 1 - G^{\Psi_N}(x,y) f(y) \right] \quad , \quad f(y) G^{\Psi_N}(y) + \frac{G^{\Psi_N}(y,x)}{G^{\Psi_N}(x)} \left[ 1 - G^{\Psi_N}(x,y) f(y) \right] = 1 \]
		and thus
		\[ f(y) \bigg( G^{\Psi_N}(y) - \frac{G^{\Psi_N}(x,y)G^{\Psi_N}(y,x)}{G^{\Psi_N}(x)} \bigg) = 1 - \frac{G^{\Psi_N}(x,y)}{G^{\Psi_N}(x)} \, .
		\]
		From this, we deduce
		\[ f(y) = \frac{1 - \frac{G^{\Psi_N}(x,y)}{G^{\Psi_N}(x)}}{G^{\Psi_N}(y) - \frac{G^{\Psi_N}(x,y)G^{\Psi_N}(y,x)}{G^{\Psi_N}(x)}} = \frac{G^{\Psi_N}(x) - G^{\Psi_N}(x,y)}{G^{\Psi_N}(x) G^{\Psi_N}(y) - G^{\Psi_N}(x,y)G^{\Psi_N}(y,x)} \, . \]
		Multiplying this by $\lambda_N \phi_N^2(y)$ gives us an expression for $e_K(y)$, however we can first notice that $\lambda_N \phi_N^2(y) = \cpc^{\Psi_N}(\{ y \}) G^{\Psi_N}(y)$. Therefore, we get
		\begin{equation}
			e_K(y) = \frac{ \cpc^{\Psi_N}(\{ y \}) G^{\Psi_N}(y)[G^{\Psi_N}(x) - G^{\Psi_N}(y,x)]}{G^{\Psi_N}(x) G^{\Psi_N}(y) - G^{\Psi_N}(x,y)G^{\Psi_N}(y,x)} \, ,
		\end{equation}
		which can be rewritten by noting that $G^{\Psi_N}(x,y) = \Pbf^{\Psi_N}_x(H_y < +\infty) G^{\Psi_N}(y)$:
		\begin{equation}
			e_K(y) = \frac{\cpc^{\Psi_N}(\{ y \}) [1 - \Pbf^{\Psi_N}_y(H_x < +\infty)]}{1 - \Pbf^{\Psi_N}_x(H_y < +\infty)\Pbf^{\Psi_N}_y(H_x < +\infty)} \, ,
		\end{equation}
		By symmetry, we also deduce the expression of $e_K(x)$, hence proving the statement since $\cpc^{\Psi_N}(\{ x,y \}) = e_K(x) + e_K(y)$ by definition.
	\end{proof}

	\begin{corollary}\label{cor:cap-2-points-LB}
		There is a constant $c_0 > 0$ such that for all $N$ large enough, any $x,y \in \Lambda_N$, we have
		\begin{equation}
			\cpc^{\Psi_N}(\mathset{x,y}) \geq \big[ 1 + \tfrac14 c_0^2 \big] \alpha_\Lambda/g(0) \, .
		\end{equation}
	\end{corollary}
	
	To prove Corollary \ref{cor:cap-2-points-LB}, we use the following result on the tilted walk.
	
	\begin{lemma}\label{lem:proba-escape-LB-c_0}
		There is a constant $c_0 > 0$ such that for all $N$ large enough,
		\begin{equation}
			\inf_{x,y \in B, x \neq y} \Pbf^{\Psi_N}_x(H_y = +\infty) \geq c_0 \, .
		\end{equation}
	\end{lemma}
	
	\begin{proof}
		Using the Markov property and \eqref{eq:encadrement-proba}, we have
		\[ \begin{split}
			\Pbf^{\Psi_N}_x(H_y = +\infty) &\geq \Pbf^{\Psi_N}_x(H_y > H_{\partial B^\eps}) \inf_{z \in \partial B^\eps} \Pbf^{\Psi_N}_x(H_y = +\infty) \\
			&\geq \kappa \Pbf_x(H_y > H_{\partial B^\eps}) \inf_{z \in \partial B^\eps} \Pbf^{\Psi_N}_x(H_y = +\infty) \, .
		\end{split} \]
		The first probability can be expressed as $1 - g(x,y)/g(0) \geq 1 - g(0,e)/g(0)$ with $|e| = 1$, and thus is bounded from below by a positive constant. The second probability is also bounded from below by a constant independent of $N$, we refer to the proof of Lemma 2.5 in \cite{bouchot2024scaling-homogene}.
	\end{proof}
	
	\begin{proof}[Proof of Corollary \ref{cor:cap-2-points-LB}]
		Note that combining Propositions \ref{prop:cap-tilt=phi2} \& \ref{prop:cap-2-points-exacte}, we only need to prove that provided $N$ large enough, we have
		\begin{equation}\label{eq:points-proches-theta/alpha}
			\frac{\Pbf^{\Psi_N}_x(H_y = +\infty) + \Pbf^{\Psi_N}_y(H_x = +\infty)}{1 - \Pbf^{\Psi_N}_x(H_y < +\infty)\Pbf^{\Psi_N}_y(H_x < +\infty)} \geq 1 + \frac14 c_0^2 \, .
		\end{equation}
		Observe that
		\begin{equation}\label{eq:points-proches-reecriture-probas-toucher}
			\begin{split}
				1 - &\Pbf^{\Psi_N}_x(H_y < +\infty)\Pbf^{\Psi_N}_y(H_x < +\infty)\\ &= \Pbf^{\Psi_N}_x(H_y = +\infty) + \Pbf^{\Psi_N}_y(H_x = +\infty) - \Pbf^{\Psi_N}_x(H_y = +\infty)\Pbf^{\Psi_N}_y(H_x = +\infty) \, .
			\end{split}
		\end{equation}
		Injecting \eqref{eq:points-proches-reecriture-probas-toucher} in the denominator of the left-hand side of \eqref{eq:points-proches-theta/alpha}, we get
		\begin{equation}
			\begin{split}
				\frac{\Pbf^{\Psi_N}_x(H_y = +\infty) + \Pbf^{\Psi_N}_y(H_x = +\infty)}{1 - \Pbf^{\Psi_N}_x(H_y < +\infty)\Pbf^{\Psi_N}_y(H_x < +\infty)}	&\geq 1 + \frac{\Pbf^{\Psi_N}_x(H_y = +\infty)\Pbf^{\Psi_N}_y(H_x = +\infty)}{\Pbf^{\Psi_N}_x(H_y = +\infty) + \Pbf^{\Psi_N}_y(H_x = +\infty)}\\
				&\geq  1 + \frac12 \Pbf^{\Psi_N}_x(H_y = +\infty)\Pbf^{\Psi_N}_y(H_x = +\infty) \, .
			\end{split}
		\end{equation}
		By Lemma \ref{lem:proba-escape-LB-c_0} we have $\Pbf^{\Psi_N}_x(H_y = +\infty)\Pbf^{\Psi_N}_y(H_x = +\infty) \geq c_0^2$. Therefore taking $N$ large enough yields \eqref{eq:points-proches-theta/alpha}, thus ending the proof.
	\end{proof}

	When the two points $x,y$ are far from each other, the capacity of $\{x,y\}$ is well-approximated by the sum of the capacities of $\{x\}$ and $\{y\}$. The bound $\cpc^{\Psi_N}(\{x,y\}) \leq \cpc^{\Psi_N}(\{x\}) + \cpc^{\Psi_N}(\{y\})$ is classical. For the lower bound, we have the following statement.

	\begin{proposition}\label{prop:cap-tilt-disjoint}
		For any disjoint $K_1,K_2 \subseteq \Lambda_N$, we have
		\begin{equation}\label{eq:LB-cap-tilt-precise}
			\cpc^{\Psi_N}(K_1 \cup K_2) \geq \cpc^{\Psi_N}(K_1) + \cpc^{\Psi_N}(K_2) - \lambda_N \sum_{z \in K_1} \sum_{w \in K_2} \left[ \phi_N^2(z) G^{\Psi_N}(z,w) + \phi_N^2(w) G^{\Psi_N}(w,z) \right] \, .
		\end{equation}
		In particular, writing $d(K_1,K_2) \defeq \inf\limits_{z \in K_1, w \in K_2} |z-w|$, there is a constant $c_{D,d} > 0$ such that
		\begin{equation}\label{eq:LB-cap-tilt-distance}
			\cpc^{\Psi_N}(K_1 \cup K_2) \geq \cpc^{\Psi_N}(K_1) + \cpc^{\Psi_N}(K_2) - c_{D,d} \frac{|K_1| \cdot |K_2|}{d(K_1,K_2)^{d-2}} \, .
		\end{equation}
	\end{proposition}
	
	\begin{proof}
		For any $z \in K_1$, note that
		\begin{equation}
			\begin{split}
				\Pbf^{\Psi_N}_z(\bar{H}_{K_1 \cup K_2} = +\infty) &= \Pbf^{\Psi_N}_z(\bar{H}_{K_1} = +\infty) - \Pbf^{\Psi_N}_z(\bar{H}_{K_1} = +\infty, \bar{H}_{K_2} < +\infty)\\
				&\geq \Pbf^{\Psi_N}_z(\bar{H}_{K_1} = +\infty) - \Pbf^{\Psi_N}_z(\bar{H}_{K_2} < +\infty) \, .
			\end{split}
		\end{equation}
		Thus, we get a lower bound using \eqref{eq:def:titled-equilibrium-mes} that reads
		\begin{equation}
			\begin{split}
				\sum_{z \in K_1} e^{\Psi_N}_{K_1 \cup K_2}(z) &\geq \lambda_N \sum_{z \in K_1} \phi_N^2(z) \left[ \Pbf^{\Psi_N}_z(\bar{H}_{K_1} = +\infty) - \Pbf^{\Psi_N}_z(\bar{H}_{K_2} < +\infty) \right]\\
				&\geq \cpc^{\Psi_N}(K_1) - \lambda_N \sum_{z \in K_1} \phi_N^2(z) \Pbf^{\Psi_N}_z(\bar{H}_{K_2} < +\infty) \, .
			\end{split}
		\end{equation}
		Using the last exit decomposition (Proposition \ref{prop:LED-green-tilt}), we deduce that
		\begin{equation}\label{eq:LB-somme-e_union-K1}
			\sum_{z \in K_1} e^{\Psi_N}_{K_1 \cup K_2}(z) \geq \cpc^{\Psi_N}(K_1) - \lambda_N \sum_{z \in K_1} \sum_{w \in K_2} \phi_N^2(z) G^{\Psi_N}(z,w) \Pbf^{\Psi_N}_w(\bar{H}_{K_2} = +\infty) \, .
		\end{equation}
		We can similarly prove
		\begin{equation}\label{eq:LB-somme-e_union-K2}
			\sum_{w \in K_2} e^{\Psi_N}_{K_1 \cup K_2}(w) \geq \cpc^{\Psi_N}(K_2) - \lambda_N \sum_{z \in K_1} \sum_{w \in K_2} \phi_N^2(w) G^{\Psi_N}(w,z) \Pbf^{\Psi_N}_z(\bar{H}_{K_1} = +\infty) \, .
		\end{equation}
		Therefore, we deduce \eqref{eq:LB-cap-tilt-precise} after using
		\[ \cpc^{\Psi_N}(K_1 \cup K_2) = \sum_{z \in K_1} e^{\Psi_N}_{K_1 \cup K_2}(z) + \sum_{w \in K_2} e^{\Psi_N}_{K_1 \cup K_2}(w) \]
		and bounding the probabilities in \eqref{eq:LB-somme-e_union-K1},\eqref{eq:LB-somme-e_union-K2} by $1$.
		To get \eqref{eq:LB-cap-tilt-distance}, it suffices to use the fact that $\phi_N$ is bounded from below (recall \eqref{eq:phiN-borne}), that $\lambda_N \leq 1$, as well as Proposition \ref{prop:encadrement-green-tilté}.
	\end{proof}

	Combining \eqref{eq:proba-vacant-tilt} with Corollary \ref{cor:cap-2-points-LB} and Proposition \ref{prop:cap-tilt-disjoint}, we get bounds on the correlations between events $\{x \not\in \mathscr{I}_{\Psi_N}(u_N) \}$ and $\{y \not\in \mathscr{I}_{\Psi_N}(u_N)\}$ which will be useful in our proofs (see \eqref{eq:formule-variance} below).
	
	\begin{proposition}[Control of the correlations]\label{prop:correlations-vacant-set}
		Let $u_N$ be such that $u_N N^{-c} \to 0$, we have the following universal bound: for distinct $x,y \in \Lambda_N$,
		\begin{equation}\label{eq:covar-univ-bound}
			\mathrm{Cov} \Big(\indic{x \not\in \mathscr{I}_{\Psi_N}(u_N)} , \indic{y \not\in \mathscr{I}_{\Psi_N}(u_N)} \Big) \leq \PP \big( x,y \not\in \mathscr{I}_{\Psi_N}(u_N) \big) \leq \exp \Big( - u_N \tfrac{\alpha_\Lambda}{g(0)} (1+\tfrac14 c_0^2) \Big) \, .
		\end{equation}
		Moreover, there is a constant $c_d^{\rho, \Lambda} > 0$ such that for any sequence $a_N \to +\infty$ and any $x,y \in \Lambda_N$ such that $|x-y| \geq a_N$, provided $N$ large enough we have
		\begin{equation}\label{eq:covar-distant-bound}
			\mathrm{Cov} \Big(\indic{x \not\in \mathscr{I}_{\Psi_N}(u_N)} , \indic{y \not\in \mathscr{I}_{\Psi_N}(u_N)} \Big) \leq 2 \exp \Big( - u_N \tfrac{\varphi_N^2(x) + \varphi_N^2(y)}{g(0)} \Big) \Big[ \exp \Big( c_d^{\Lambda} \frac{u_N}{a_N^{d-2}} \Big) - 1 \Big]
		\end{equation}
	\end{proposition}
	
	\begin{proof}
		The first inequality is a direct consequence of Corollary \ref{cor:cap-2-points-LB} applied to
		\begin{equation}
			\PP \big( x,y \not\in \mathscr{I}_{\Psi_N}(u_N) \big) = \exp \Big( - u_N \cpc^{\Psi_N}(\{ x,y \} \big) \Big) \, .
		\end{equation}
		For the second inequality, we use Proposition \ref{prop:cap-tilt-disjoint}-\eqref{eq:LB-cap-tilt-distance} to get
			\begin{equation}
				\cpc^{\Psi_N}(\{x,y\}) \geq \cpc^{\Psi_N}(\{x\}) + \cpc^{\Psi_N}(\{x\}) - \frac{c_d^\Lambda}{|x-y|^{d-2}} \, .
			\end{equation}
			Now, using the formula \eqref{eq:proba-vacant-tilt}, we get
			\begin{equation}
				\PP \big( x,y \not\in \mathscr{I}_{\Psi_N}(u_N) \big) \leq \PP \big( x \not\in \mathscr{I}_{\Psi_N}(u_N) \big) \PP \big( y \not\in \mathscr{I}_{\Psi_N}(u_N) \big) \exp \Big(u_N \frac{c_d^\Lambda}{|x-y|^{d-2}} \Big) \, ,
			\end{equation}
			thus proving the bound after using \eqref{eq:proba-vacant-tilt} and Proposition \ref{prop:cap-tilt=phi2} to replace $\cpc^{\Psi_N}$ by $\varphi_N^2/g(0)$ (with a factor $2$ taking into account the error of this replacement).
	\end{proof}

	\section{Covering level of $\Lambda_N$ - study of the intermediary set}\label{sec:caracterisation-Lambda-rho}
	
	\subsection{Some technical considerations}

	The general principle is the same as Belius' approach for the covering level on standard random interlacements. It relies on the following identity, which is a direct consequence of the Poisson structure of random interlacements: let $K \subset \ZZ^d$ be a finite set and $u_1 < u_2$, then for any $K' \subset K$,
	\begin{equation}\label{eq:covering-RI-sequencial-indep}
		\PP \big( K \subseteq \mathscr{I}_{\Psi_N}(u_2) \big) = \PP \big( K \setminus K' \subseteq \mathscr{I}_{\Psi_N}(u_2 - u_1) \big) \cdot \PP \big( K \cap \mathscr{I}_{\Psi_N}(u_1) = K' \big) \, .
	\end{equation}
	The key point that Belius' approach exploits is that near the covering time of $\Lambda_N$, the set of points that are yet to be covered is ``well-separated'', which makes it easier to study thanks to decorrelation inequalities.
	
	In the rest of the paper, we fix $\rho > 0$ small enough and define the intermediary set of ``late points'' as
	\begin{equation}
		\Lambda_N(\rho) \defeq \big\{ x \in \Lambda_N \, : \, x \not\in \mathscr{I}_{\Psi_N}(u_N^{\Lambda,\rho}) \big\} \quad \text{with} \quad u_N^{\Lambda,\rho} \defeq (1-\rho) \tfrac{g(0)}{\alpha_\Lambda} \log |\Lambda_N| \, .
	\end{equation}
	Applying \eqref{eq:covering-RI-sequencial-indep} then yields
	\begin{equation}
		\PP \big( \mathfrak{U}_N^{\Lambda} \leq u_N^{\Lambda}(z) \, \big| \, \Lambda_N(\rho) \big) = \PP \big( \Lambda_N(\rho) \subseteq \mathscr{I}_{\Psi_N}(u_N^{\Lambda}(z) - u_N^{\Lambda,\rho}) \big) \, .
	\end{equation}
	
	In order to get some estimates, we will often use two results in combination in order to ``integrate'' a function on $\Lambda_N(\rho)$. Let $f : \Lambda_N \longrightarrow \RR$, we first have
	\begin{equation}\label{eq:formule-variance}
		\mathrm{Var} \Big[ \sum_{z \in \Lambda_N(\rho)} f(z) \Big] = \sum_{x,y \in \Lambda_N} f(x)f(y) \mathrm{Cov} \big(\indic{x \in \Lambda_N(\rho)} , \indic{y \in \Lambda_N(\rho)} \big) \, .
	\end{equation}
	The proof of \eqref{eq:formule-variance} is straightforward. Coupled with Proposition \ref{prop:correlations-vacant-set}, this allows us to get upper bounds on the variance of such ``integral'' without consideration for the actual random set $\Lambda_N(\rho)$.
	
	The following lemma states a useful asymptotic that we use to fully exploit \eqref{eq:formule-variance} in the case where $f$ is either constant or an exponential. We postpone its proof to Appendix \ref{app:cv-somme-exp}, as it is quite technical and relies on Assumption \ref{hyp:deformation-level-sets}.
	
	\begin{lemma}\label{lem:sum-exp-varphi}
		Fix $\beta > 0$, then under Assumption \ref{hyp:deformation-level-sets} we have the convergence
		\begin{equation}\label{eq:lemme-technique-integrale}
			\lim_{N \to +\infty} \frac{\log |\Lambda_N|}{|\Lambda_N|^{1-\beta}} \sum_{x \in \Lambda_N} \exp \Big( - \beta \frac{\varphi_N^2(x)}{\alpha_\Lambda} \log |\Lambda_N| \Big) = \frac{\kappa_\Lambda}{\beta} \, .
		\end{equation}
		Moreover, the convergence still holds if $\log |\Lambda_N|$ is replaced by $(1+\eps_N)\log |\Lambda_N|$ with $(\eps_N)_{N \geq 1}$ a vanishing sequence.
	\end{lemma}
	
	The sum in \eqref{eq:lemme-technique-integrale} naturally appears due to \eqref{eq:proba-vacant-tilt} and the fact that $\varphi_N^2(x) \approx \cpc^{\Psi_N}(\{x\})$ (recall Proposition \ref{prop:cap-tilt=phi2}).

	\subsection{Scattering of the late point}

	\begin{proposition}\label{prop:scattering-lambda-rho}
		Let $\rho > 0$ be small enough and define $a_N^\rho \defeq |\Lambda_N|^{\frac{4\rho}{d-2}}$, then
		\begin{equation}
			\sum_{0 < |x-y| \leq a_N^\rho} \PP \big( x,y \in \Lambda_N(\rho) \big) \lesssim |\Lambda_N|^{-\rho} \, . 
		\end{equation}
		This holds whether or not Assumption \ref{hyp:deformation-level-sets} is satisfied.
	\end{proposition}
	
	\begin{proof}
		We separate the sum depending on whether $|x-y| \leq (\log |\Lambda_N|)^2$ or not. For very close $x$ and $y$, the universal bound of Corollary \ref{cor:cap-2-points-LB} suffices and we get
		\begin{equation}
			\sum_{0<|x-y|\leq (\log |\Lambda_N|)^2} \PP \big( x,y \in \Lambda_N(\rho) \big) \leq c (\log |\Lambda_N|)^{2d} |\Lambda_N| \exp \big( - (1-\rho)(1 + \tfrac14 c_0^2)\log |\Lambda_N| \big) \, ,
		\end{equation}
		which is equal to $c (\log |\Lambda_N|)^{2d} |\Lambda_N|^{\frac14c_0^2 - \rho(1 + \frac14 c_0^2)}$, which is less than $c' |\Lambda_N|^{-\rho}$ provided $\rho > 0$ small enough.
		
		For the $x,y$ that are farther away, we instead use Proposition \ref{prop:cap-tilt-disjoint}-\eqref{eq:LB-cap-tilt-distance} to get
		\begin{equation}
			\cpc^{\Psi_N}(\mathset{x,y}) \geq \cpc^{\Psi_N}(\mathset{x}) + \cpc^{\Psi_N}(\mathset{y}) - \frac{c_d}{(\log |\Lambda_N|)^{2d-4}} \, .
		\end{equation}
		Therefore, also using Proposition \ref{prop:cap-tilt=phi2} we deduce 
		\begin{equation}
			\PP \big( x,y \in \Lambda_N(\rho) \big) \leq (1+\bar{o}(1)) \exp \Big( -(1-\rho) \tfrac{\varphi_N^2(x) + \varphi_N^2(y)}{\alpha_\Lambda} \log |\Lambda_N| + \frac{c_d}{(\log |\Lambda_N|)^{2d-5}}\Big) \, .
		\end{equation}
		Again, as $d \geq 3$ we have $(\log |\Lambda_N|)^{2d-5} \to +\infty$ and thus
		\begin{equation}
			\sum_{(\log |\Lambda_N|)^2 < |x-y| \leq a_N^\rho} \PP \big( x,y \in \Lambda_N(\rho) \big) \leq c \sum_{(\log |\Lambda_N|)^2 < |x-y| \leq a_N^\rho} \exp \Big( -(1-\rho) \tfrac{\varphi_N^2(x) + \varphi_N^2(y)}{\alpha_\Lambda} \log |\Lambda_N| \Big) \, .
		\end{equation}
		Since $\varphi_N^2(y) \geq \alpha_\Lambda$ by definition, we get that this is less than
		\begin{equation}
			c \frac{(a_N^\rho)^d}{|\Lambda_N|^{1-\rho}} \sum_{x \in \Lambda_N} \exp \Big( -(1-\rho) \tfrac{\varphi_N^2(x)}{\alpha_\Lambda} \log |\Lambda_N| \Big)
			= c' \frac{(a_N^\rho)^d}{|\Lambda_N|^{1-\rho}} (1+\bar{o}(1)) \frac{\kappa_\Lambda}{1-\rho} \frac{|\Lambda_N|^\rho}{\log |\Lambda_N|} \, ,
		\end{equation}
		where we used Lemma \ref{lem:sum-exp-varphi} for the equality.
		Injecting the definition of $a_N^\rho$, we see that provided $N$ large enough and $\rho$ small enough so that $1 - \rho > \tfrac{4 \rho d}{d-2} + 2\rho$, this is less that some constant times $|\Lambda_N|^{-\rho}$, hence proving  the proposition. Note that instead of using Lemma \ref{lem:sum-exp-varphi}, we could also note that $\varphi_N^2(x) \geq \alpha_\Lambda$ which provides a bound $c'' (a_N^\rho)^d|\Lambda_N|^{2\rho - 1} = c'' |\Lambda_N|^{2\rho + \frac{4 \rho d}{d-2} - 1} \leq c'' |\Lambda_N|^{-\rho}$ provided $1 - 2\rho > \tfrac{4 \rho d}{d-2} + 2\rho$. Thus, the proposition also holds in the general case (without Assumption \ref{hyp:deformation-level-sets}.
	\end{proof}

	\subsection{Cardinality of the set of late points}

	\begin{proposition}\label{prop:borne-cardinal-Lambda-rho}
		Under Assumption \ref{hyp:deformation-level-sets} and provided $\rho > 0$ small enough, we have the following equivalence in probability:
		\begin{equation}
			|\Lambda_N(\rho)| \sim \frac{\kappa_\Lambda}{1-\rho} \frac{|\Lambda_N|^{\rho}}{\log  |\Lambda_N|}
		\end{equation}
	\end{proposition}
	
	\begin{proof}
		First observe that according to \eqref{eq:proba-vacant-tilt} and Corollary \ref{cor:sum-cap-sum-varphi}:
		\begin{equation}
			\EE \big[ |\Lambda_N(\rho)| \big] = \sum_{x \in \Lambda_N} \PP \big( x \not\in \mathscr{I}_{\Psi_N}(u_N^\rho) \big) = (1 + \bar{o}(1)) \sum_{x \in \Lambda_N} \exp \Big( - (1-\rho) \frac{\varphi_N^2(x)}{\alpha_\Lambda} \log |\Lambda_N| \Big) \, .
		\end{equation}
		Now, using Lemma \ref{lem:sum-exp-varphi}, we get that this is asymptotically $\kappa_\Lambda |\Lambda_N|^\rho/(1-\rho)\log |\Lambda_N|$.
		
		Let us now prove that $\mathrm{Var}[|\Lambda_N(\rho)|]$ is at most of order $\EE \big[ |\Lambda_N(\rho)| \big]$, which will prove the proposition using Chebychev inequality.
		
		Applying \eqref{eq:formule-variance}, we have
		\begin{equation}
			\mathrm{Var}[|\Lambda_N(\rho)|] = \sum_{x,y \in \Lambda_N} \mathrm{Cov} \big(\indic{x \in \Lambda_N(\rho)} , \indic{y \in \Lambda_N(\rho)} \big) \, .
		\end{equation}
		Let us consider separately three cases in order to control the covariances.
		
		First, for $x=y$, this covariance is less than $\PP (x \in \Lambda_N(\rho))$ and thus
		\begin{equation}
			\sum_{x=y \in \Lambda_N} \mathrm{Cov} \big(\indic{x \in \Lambda_N(\rho)} , \indic{y \in \Lambda_N(\rho)} \big) \leq \sum_{x \in \Lambda_N} \PP (x \in \Lambda_N(\rho)) = \EE \big[ |\Lambda_N(\rho)| \big] \, .
		\end{equation}
		
		Then, assume $0 < |x-y| \leq a_N^\rho$ and use the previous Proposition \ref{prop:scattering-lambda-rho} to get
		\begin{equation}
			\sum_{\substack{x,y \in \Lambda_N\\ 0 < |x-y| \leq a_N^\rho}} \mathrm{Cov} \big(\indic{x \in \Lambda_N(\rho)} , \indic{y \in \Lambda_N(\rho)} \big) \leq \sum_{\substack{x,y \in \Lambda_N\\ 0 < |x-y| \leq a_N^\rho}} \PP (x,y \in \Lambda_N(\rho)) \lesssim |\Lambda_N|^{-\rho} \, .
		\end{equation}
		
		Finally, assume $|x-y| > a_N^\rho$ in which case we use the bound \eqref{eq:covar-distant-bound} to get
		\begin{equation}
			\sum_{\substack{x,y \in \Lambda_N\\ |x-y| > a_N^\rho}} \mathrm{Cov} \big(\indic{x \in \Lambda_N(\rho)} , \indic{y \in \Lambda_N(\rho)} \big) \leq \sum_{\substack{x,y \in \Lambda_N\\ |x-y| > a_N^\rho}} \exp \Big( - (1-\rho) \tfrac{\varphi_N^2(x) + \varphi_N^2(y)}{\alpha_\Lambda} \Big) \Big[ \exp \Big( c_d^{\Lambda} \frac{u_N^\rho}{a_N^{d-2}} \Big) - 1 \Big] \, .
		\end{equation}
		Writing this double sum as a product, since $u_N^\rho \ll a_N^{d-2}$, provided $N$ large enough, we get
		\begin{equation}
			\sum_{\substack{x,y \in \Lambda_N\\ |x-y| > a_N^\rho}} \mathrm{Cov} \big(\indic{x \in \Lambda_N(\rho)} , \indic{y \in \Lambda_N(\rho)} \big) \leq 2 c_d^{\Lambda} \frac{u_N^\rho}{a_N^{d-2}} \Big( \sum_{x \in \Lambda_N} \exp \big( -(1-\rho) \tfrac{\varphi_N^2(x)}{\alpha_\Lambda} \log |\Lambda_N| \big) \Big)^2 \, .
		\end{equation}
		Therefore, using Lemma \ref{lem:sum-exp-varphi}, we get
		\begin{equation}
			\sum_{\substack{x,y \in \Lambda_N\\ |x-y| > a_N^\rho}} \mathrm{Cov} \big(\indic{x \in \Lambda_N(\rho)} , \indic{y \in \Lambda_N(\rho)} \big) \lesssim \frac{\log |\Lambda_N|}{|\Lambda_N|^{2\rho}} \times \frac{|\Lambda_N|^{4\rho}}{\log^2 |\Lambda_N|} =\frac{|\Lambda_N|^{-2\rho}}{\log |\Lambda_N|} \, .
		\end{equation}
		Combining all of the above yields $\mathrm{Var}[|\Lambda_N(\rho)|] \leq c \EE \big[ |\Lambda_N(\rho)| \big]$ for $N$ large enough, hence the asymptotics $|\Lambda_N(\rho)| \sim \EE \big[ |\Lambda_N(\rho)| \big]$ in probability using Chebychev's inequality.
	\end{proof}
	
	\begin{corollary}\label{cor:late-points-card-Lambda-general}
		Provided $\rho, \delta > 0$ small enough, there is a $c_{\rho,\delta} > 0$ such that for all $N$ large enough, with probability at least $1 - c_{\rho,\delta} |\Lambda_N|^{-\rho/4}$ we have
		\begin{equation}
			|\Lambda_N|^{\rho - \delta(1-\rho)} \leq |\Lambda_N(\rho)| \leq |\Lambda_N|^{\rho + N^{-c}} + |\Lambda_N|^{2\rho/3} \, .
		\end{equation}
	\end{corollary}
	
	\begin{proof}
		With again the trivial bound $\cpc^{\Psi_N}(\{x\}) \geq \varphi_N^2(x)(1 - N^{-c}) \geq \alpha_\Lambda(1 - N^{-c})$ on $\Lambda_N$ (recall Proposition \ref{prop:cap-tilt=phi2}), we get $\EE[|\Lambda_N(\rho)|] \leq |\Lambda_N|^{\rho + N^{-c}(1-\rho)}$. For the lower bound, we get however that for any $\delta > 0$,
		\begin{equation}
			\esp{|\Lambda_N(\rho)|} \geq |\Lambda_N|^{-(1-\rho)(1+\delta)} \Big| \big\{ x \in \Lambda_N \, : \, \cpc^{\Psi_N}(\mathset{x}) \leq (1+\delta)\alpha_\Lambda/G(0) \big\} \Big| \, .
		\end{equation}
		Using the regularity of $\phi_N$ and Proposition \ref{prop:cap-tilt=phi2}, we conclude that there exists a positive constant $c_{\Lambda,\delta}$ such that $\Big| \mathset{x \in \Lambda_N \, : \, \cpc^{\Psi_N}(\{x\}) \leq (1+\delta)\alpha(\Lambda)/G(0)} \Big| \geq c_{\Lambda,\delta} |\Lambda_N|$. Therefore, we get that for any $\delta > 0$, we have $\esp{|\Lambda_N(\rho)|} \geq |\Lambda_N|^{\rho - \delta(1-\rho)}$.
		We then show without difficulty that $\Var{|\Lambda_N(\rho)|} \leq c_\rho^{\Lambda} |\Lambda_N|^{\rho + \delta_N(1-\rho)}$ using the same proof as Proposition \ref{prop:borne-cardinal-Lambda-rho} (and the bound $\varphi_N^2 \geq \alpha_\Lambda$) and concluding using Chebychev inequality.
	\end{proof}

	\section{Covering by interlacements and Poisson limit}

	\subsection{Fluctuations of the covering level}

	Let us first deal with the fluctuations of the covering time. We begin with the proof of Theorem \ref{th:cover-lambda-entrelac}-\eqref{eq:cover-level-Lambda-asymptotics}, that is the asymptotics $\mathfrak{U}_N(\Lambda_N) \sim g(0) \alpha_\Lambda^{-1} \log |\Lambda_N|$ in probability. We recall that this asymptotics does not require the Assumption \ref{hyp:deformation-level-sets} to hold.
	
	Recall that we defined $\Lambda_N(\rho)$ as the set of points that are not covered at the level $u_{N,\Lambda}^\rho = (1-\rho) \tfrac{g(0)}{\alpha_\Lambda} \log |\Lambda_N|$.
	
	\begin{proof}[Proof of Theorem \ref{th:cover-lambda-entrelac}-\eqref{eq:cover-level-Lambda-asymptotics} (first order asymptotics)]
		Write $u_N^\Lambda \defeq g(0) \alpha_\Lambda^{-1} \log |\Lambda_N|$ and fix some $\delta > 0$.
		We first notice that using Corollary \ref{cor:late-points-card-Lambda-general},
		\begin{equation}
			\PP \big( \mathfrak{U}_N(\Lambda_N) \leq (1-\delta) u_N^\Lambda \big) = \PP \big( \Lambda_N(\delta) = \varnothing \big) \leq \PP \big( |\Lambda_N(\delta)| \leq |\Lambda_N|^{\delta/8} \big) \leq |\Lambda_N|^{-\delta/4} \, .
		\end{equation}
		On the other hand, similarly to the proof of the previous lemma, we have
		\begin{equation}
			\PP \big( \mathfrak{U}_N(\Lambda_N) \geq (1+\delta) u_N(\Lambda) \big) \leq (1+\bar{o}(1)) \sum_{x \in \Lambda_N} e^{- \frac{\varphi_N^2(x)}{\alpha_\Lambda} (1+\eps) \log |\Lambda_N|} \leq |\Lambda_N|^{-\delta}
		\end{equation}
		Therefore, the two probabilities go to zero, hence proving $\mathfrak{U}_N(\Lambda_N) \sim u_N^\Lambda$ in probability.
	\end{proof}
	
	If we want to get more precise results, we need to better understand the covering of the late points.
	As previously stated, the proof of Theorems \ref{th:cover-lambda-entrelac} \& \ref{th:gumbel+poisson-RI} relies on the fact that $\Lambda_N(\rho)$ is very sparse, and thus each of its points are covered independently from the others.

	We define the following ``good'' event for $\Lambda_N(\rho)$:
	\begin{equation}
		\mathcal{A}_{N,\Lambda}^\rho \defeq \Big\{ |\Lambda_N(\rho)| \leq c_\rho |\Lambda_N|^{\rho} \; ; \; \inf_{x,y \in \Lambda_N(\rho), x \neq y} |x-y| \geq a_N^\rho \Big\} \, .
	\end{equation}

	\begin{lemma}\label{lem:cover-level-Z_N}
		Let $u = u(N)$ that grows logarithmically in $N$. We have the \textit{a.s.} convergence
		\begin{equation}
			\lim_{N \to +\infty} \Big| \PP \Big( \mathfrak{U}_N^{\Lambda} \leq u_N^{\Lambda, \rho} + u \, \Big| \, \Lambda_N(\rho) \Big) - \exp \Big( - \sum_{x \in \Lambda_N(\rho)} \exp \big( - \tfrac{u}{g(0)} \varphi_N^2(x) \big) \Big) \Big| \mathbbm{1}_{\mathcal{A}_{N,\Lambda}^\rho} = 0 \, .
		\end{equation}
	\end{lemma}
	
	\begin{proof}
		Recall \eqref{eq:covering-RI-sequencial-indep}: by the Poisson structure of the tilted RI, conditioning on $\Lambda_N(\rho)$ we get
		\begin{equation}\label{eq:loi-covering-cond-Lambda-rho}
			\PP \Big( \mathfrak{U}_N^{\Lambda} \leq u_N^{\Lambda}(z) \, \Big| \, \Lambda_N(\rho) \Big) = \PP \Big( \Lambda_N(\rho) \subseteq \mathscr{I}_{\Psi_N}(u) \Big) \, ,
		\end{equation}
		where the $\PP$ is only on $\mathscr{I}_{\Psi_N}(u)$ and $\Lambda_N(\rho)$ is now a fixed set.
		We now turn to the study of the probability on the right-hand side of \eqref{eq:loi-covering-cond-Lambda-rho}. Since we are on $\mathcal{A}_{N,\Lambda}^\rho$, applying Proposition \ref{prop:decouplage-recouvrement} gives
		\begin{equation}\label{eq:erreur-decouplage-late-Lambda}
			\Big| \PP \Big( \Lambda_N(\rho) \subset \mathscr{I}_{\Psi_N}(u) \Big) - \prod_{x \in \Lambda_N(\rho)} \PP \big(x \in \mathscr{I}_{\Psi_N}(u) \big) \Big| \mathbbm{1}_{\mathcal{A}_{N,\Lambda}^\rho} \leq c \log |\Lambda_N| \frac{|\Lambda_N|^\rho}{(a_N^\rho)^{d-2}} \, ,
		\end{equation}
		where we used $u \leq c \log |\Lambda_N|$. Recall that $a_N^\rho = |\Lambda_N|^{\frac{4\rho}{d-2}}$, hence the right-hand side of \eqref{eq:erreur-decouplage-late-Lambda} goes to $0$ as $N \to +\infty$.
		
		Let us now investigate the product of probabilities over $x \in \Lambda_N(\rho)$, that we can rewrite
		\begin{equation*}
			\exp \Big( \sum_{x \in \Lambda_N(\rho)} \log \Big[ 1 - e^{- u \cpc^{\Psi_N}(\{x\})} \Big] \Big) = \exp \Big( - \sum_{x \in \Lambda_N(\rho)} \exp \big( - u \varphi_N^2(x) (1 + \grdO(N^{-c} \vee e^{-C u})\big)  \Big] \Big) \, ,
		\end{equation*}
		with $\grdO(N^{-c} \vee e^{-Cu}) = \bar{o}(|\Lambda_N|^{-c'})$ that is deterministic and uniform in $x \in \Lambda_N$.
	\end{proof}
	
	Let us first consider the general case (that is without Assumption \ref{hyp:deformation-level-sets}), which is easily handled using Lemma \ref{lem:cover-level-Z_N}.
	
	\begin{proof}[Proof of Theorem \ref{th:cover-lambda-entrelac}-\eqref{eq:super-gumbel-general-RI}]
		We use Lemma \ref{lem:cover-level-Z_N} with $u = \tfrac{g(0)}{\alpha_\Lambda} \big\{ \rho \log |\Lambda_N| + z \big\}$: for any $\eta > 0$, provided $N$ large enough,
		\begin{equation}
			\begin{split}
				\PP \Big( \mathfrak{U}_N^{\Lambda} \leq u_N^{\Lambda, \rho} + u \, \Big| \, \Lambda_N(\rho) \Big) \mathbbm{1}_{\mathcal{A}_{N,\Lambda}^\rho} \geq (1-\delta) \exp \Big( - |\Lambda_N(\rho)| |\Lambda_N|^{-\rho} e^{-z} \Big) \mathbbm{1}_{\mathcal{A}_{N,\Lambda}^\rho} \, .
			\end{split}
		\end{equation}
		where we also used $\varphi_N^2(x) \geq \alpha_\Lambda$ for all $x \in \Lambda_N$.
		Now, since we can further restrict ourselves to the event $|\Lambda_N(\rho)| \leq |\Lambda_N|^{\rho + N^{-c}}$ (recall Corollary \ref{cor:late-points-card-Lambda-general}), we get that if $N$ is large enough, this is greater than $(1-2\delta) \exp(-e^{-z}) \mathbbm{1}_{\mathcal{A}_{N,\Lambda}^\rho}$. Since $\PP(\mathcal{A}_{N,\Lambda}^\rho) \to 1$, we finally get \eqref{eq:super-gumbel-general-RI}.
	\end{proof}

	We now turn to the case where $\Lambda$ satisfies Assumption \ref{hyp:deformation-level-sets}. Fix $z \in \RR$ and apply Lemma \ref{lem:cover-level-Z_N} with $u = u_N^\Lambda(z) - u_N^{\Lambda,\rho}$ so that we are left to study
	\begin{equation}\label{eq:javoue-jai-plus-didee}
		\sum_{x \in \Lambda_N(\rho)} \exp \Big( - \frac{\varphi_N^2(x)}{\alpha_\Lambda} \big[ \rho \log |\Lambda_N| - \log\log |\Lambda_N| + z \big] \Big)
	\end{equation}
	
	We will use the following proposition.	
	
	\begin{proposition}\label{prop:convergence-Z_N-rho}
		Fix $\rho > 0$ small enough, then under Assumption \ref{hyp:deformation-level-sets} we have the convergence in probability
		\begin{equation}
			Z_{N,\Lambda}^\rho \defeq \frac{\log |\Lambda_N|}{|\Lambda_N|^\rho} \sum_{x \in \Lambda_N(\rho)} \exp \Big( - \rho \big[ \tfrac{\varphi_N^2(x)}{\alpha_\Lambda} - 1 \big] \log |\Lambda_N|  \Big) \xrightarrow[N \to +\infty]{\PP} \kappa_\Lambda \, .
		\end{equation}
		Moreover, this limit also holds if we replace $\log |\Lambda_N|$ by $(1+\eps_N)\log |\Lambda_N|$ where $\eps_N \to 0$.
	\end{proposition}

	\begin{proof}[Proof of Theorem \ref{th:gumbel+poisson-RW}-\eqref{eq-th:gumbel-RI}.]
	Writing $\eps_N = (- \log\log |\Lambda_N| + z)/\rho \log |\Lambda_N|$, we can rewrite the sum in \eqref{eq:javoue-jai-plus-didee} as
	\begin{equation}\label{eq:cv-produit-proba-Lambda}
		e^{-z} \frac{\log |\Lambda_N|}{|\Lambda_N|^{\rho}} \sum_{x \in \Lambda_N(\rho)} \exp \Big( - \rho \big[ \tfrac{\varphi_N^2(x)}{\alpha_\Lambda} - 1 \big](1+\eps_N) \log |\Lambda_N| \Big) \xrightarrow[N \to +\infty]{\PP} \kappa_\Lambda e^{-z} \, ,
	\end{equation}
	where we used Proposition \ref{prop:convergence-Z_N-rho}.
	Therefore, injecting \eqref{eq:cv-produit-proba-Lambda} in Lemma \ref{lem:cover-level-Z_N}, we get
	\begin{equation}
		\Big| \PP \big( \mathfrak{U}_N^{\Lambda} \leq u_N^{\Lambda}(z) \, \big| \, \Lambda_N(\rho) \big) - \exp \big( -  \kappa_\Lambda e^{-z} \big) \Big| \mathbbm{1}_{\mathcal{A}_{N,\Lambda}^\rho} \xrightarrow[N \to +\infty]{\PP} 0 \, .
	\end{equation}
	Since $\PP(\mathcal{A}_{N,\Lambda}^\rho) \to 1$, we get the desired convergence $\PP \big( \mathfrak{U}_N^{\Lambda} \leq u_N^{\Lambda}(z)  \, \big| \, \Lambda_N(\rho) \big) \to \exp \big( -  \kappa_\Lambda e^{-z} \big)$ in $L^1$ by dominated convergence. The theorem follows immediately.
	\end{proof}

	\begin{proof}[Proof of Proposition \ref{prop:convergence-Z_N-rho}]
		We first notice that we can rewrite $Z_{N,\Lambda}^\rho$ as
		\begin{equation}
			Z_{N,\Lambda}^\rho = \frac{\log |\Lambda_N|}{|\Lambda_N|^\rho} \sum_{x \in \Lambda_N} \exp \Big( - \rho \big[ \tfrac{\varphi_N^2(x)}{\alpha_\Lambda} - 1 \big] \log |\Lambda_N|  \Big) \indic{x \in \Lambda_N(\rho)} \, .
		\end{equation}
		Thus, we can easily compute its expectation using the explicit formula \eqref{eq:proba-vacant-tilt} as well as Lemma \ref{lem:sum-exp-varphi} with $\beta = 1$. We get
		\begin{equation}
			\lim_{N \to +\infty} \EE \big[ Z_{N,\Lambda}^\rho \big] = \lim_{N \to +\infty} \log |\Lambda_N| \sum_{x \in \Lambda_N} \exp \Big( -  \tfrac{\varphi_N^2(x)}{\alpha_\Lambda} \log |\Lambda_N| \Big) = \kappa_\Lambda \, .
		\end{equation}
		Let us now prove that $\mathrm{Var} \big[ Z_{N,\Lambda}^\rho \big] \to 0$, which will prove the proposition.
		
		Using the formula \eqref{eq:formule-variance}, we write
		\begin{equation}
			\mathrm{Var} \big[ Z_{N,\Lambda}^\rho \big] = (\log |\Lambda_N|)^2 \sum_{x,y \in \Lambda_N} \exp \Big( - \rho \tfrac{\varphi_N^2(x) + \varphi_N^2(y)}{\alpha_\Lambda} \log |\Lambda_N| \Big) \mathrm{Cov} \big(\indic{x \in \Lambda_N(\rho)} , \indic{y \in \Lambda_N(\rho)} \big) \, .
		\end{equation}
		We will subdivize this sum into three different parts: $x=y$; close but distinct $x,y$; and $x,y$ that are far away from each other.
		
		\textbf{Control for $x=y$:} here we consider the sum
		\begin{equation}
			\sum_{x \in \Lambda_N} \exp \Big( - 2\rho \tfrac{\varphi_N^2(x)}{\alpha_\Lambda} \log |\Lambda_N|  \Big) \Big[ \PP \big(x \in \Lambda_N(\rho) \big) - \PP \big(x \in \Lambda_N(\rho) \big)^2 \Big]
		\end{equation}
		which we bound from above by
		\begin{equation}
			\sum_{x \in \Lambda_N} \exp \Big( - \big[ 2\rho \tfrac{\varphi_N^2(x)}{\alpha_\Lambda} + (1-\rho) \tfrac{\varphi_N^2(x)}{\alpha_\Lambda} \big] \log |\Lambda_N| \Big)
			= \sum_{x \in \Lambda_N} \exp \Big( - (\rho+1) \tfrac{\varphi_N^2(x)}{\alpha_\Lambda} \log |\Lambda_N| \Big) \, .
		\end{equation}
		Using Lemma \ref{lem:sum-exp-varphi}, we get
		\begin{equation}
			\lim_{N \to +\infty} (\log |\Lambda_N|)^2 \mathcal{I}_{N,\rho}^0 = \lim_{N \to +\infty} (\log |\Lambda_N|) |\Lambda_N|^{-\rho} = 0 \, .
		\end{equation}
		
		\textbf{Control for close $x,y$:}
		We first use the universal bound given by Corollary \ref{cor:cap-2-points-LB} on the $x,y$ such that $0 < |x-y| \leq \log^2 |\Lambda_N|$ as well as Lemma \ref{lem:sum-exp-varphi} to get
		\begin{equation}
			\begin{split}
				\sum_{\substack{x,y \in \Lambda_N\\ 0 < |x-y| \leq \log^2 |\Lambda_N|}} \exp \Big( - \rho \tfrac{\varphi_N^2(x) + \varphi_N^2(y)}{\alpha_\Lambda} & \log |\Lambda_N| \Big) \mathrm{Cov} \big(\indic{x \in \Lambda_N(\rho)} , \indic{y \in \Lambda_N(\rho)} \big) \\
				&\leq c_1 \frac{\log^2 |\Lambda_N|}{|\Lambda_N|^{(1-\rho)(1+\tfrac14 c_0^2)}} \frac{|\Lambda_N|^{1-\rho}}{\log |\Lambda_N|} = \frac{c_1 \log |\Lambda_N|}{|\Lambda_N|^{(1-\rho)\tfrac14 c_0^2}} \, .
			\end{split}
		\end{equation}
		Recall that $a_N^\rho = |\Lambda_N|^{\frac{4 d \rho}{d-2}}$. We then turn to the $x,y \in \Lambda_N$ such that $(\log |\Lambda_N|)^2 < |x-y| < a_N^\rho$. Using the bound \eqref{eq:covar-distant-bound} as well as $\varphi_N^2(y) \geq \alpha_\Lambda$, we get
		\begin{equation}
			\begin{split}
				&\sum_{\substack{x,y \in \Lambda_N\\ \log^2 |\Lambda_N| < |x-y| \leq a_N}} \exp \Big( - \rho \tfrac{\varphi_N^2(x) + \varphi_N^2(y)}{\alpha_\Lambda} \log |\Lambda_N| \Big) \mathrm{Cov} \big(\indic{x \in \Lambda_N(\rho)} , \indic{y \in \Lambda_N(\rho)} \big) \\
				&\leq c_2 \frac{(a_N^\rho)^d \wedge |\Lambda_N|}{|\Lambda_N|} \sum_{x \in \Lambda_N} \exp \Big( - \tfrac{\varphi_N^2(x)}{\alpha_\Lambda} \log |\Lambda_N| \Big) \Big[ 1 - \exp \Big( c_d^{\Lambda,\rho} (\log |\Lambda_N|)^{5-2d} \Big) \Big] \, .
			\end{split}
		\end{equation}
		Using Lemma \ref{lem:sum-exp-varphi} and $d \geq 3$, this is bounded by $c_2 |\Lambda_N|^{\frac{4 d \rho}{d-2}-1} / \log^2 |\Lambda_N| \leq c_2' |\Lambda_N|^{-\rho}$ provided $\rho > 0$ small enough.
		
		\textbf{Control for distant $x,y$:} we use again \eqref{eq:covar-distant-bound}:
		\begin{equation}
			\begin{split}
				\sum_{\substack{x,y \in \Lambda_N\\ |x-y| > a_N}} \exp \Big( - \rho \tfrac{\varphi_N^2(x) + \varphi_N^2(y)}{\alpha_\Lambda} \log |\Lambda_N| \Big) &\mathrm{Cov} \big(\indic{x \in \Lambda_N(\rho)} , \indic{y \in \Lambda_N(\rho)} \big) \\
				&\leq c \bigg( \sum_{x \in \Lambda_N} \exp \Big( - \tfrac{\varphi_N^2(x)}{\alpha_\Lambda} \log |\Lambda_N| \Big) \bigg)^2 \frac{c' \log |\Lambda_N|}{(a_N^\rho)^{d-2}} \, .
			\end{split}
		\end{equation}
		Again, using Lemma \ref{lem:sum-exp-varphi} and plugging in $a_N^\rho = |\Lambda_N|^{\tfrac{4\rho}{d-2}}$, this sum is bounded from above by $c_3 (\log |\Lambda_N|)^3 |\Lambda_N|^{-4\rho}$ for some positive $c_3$.
		
		Combining all the above, we see that there exists a $c > 0$ such that $\mathrm{Var} \big[ Z_{N,\Lambda}^\rho \big]$ is at most of order $(\log |\Lambda_N|)^5 |\Lambda_N|^{-c} \to 0$, hence completing the proof.
	\end{proof}

	\subsection{Poisson limit}

	We now give a precise statement regarding the "late points" and their convergence towards a Poisson point process.
	
	Fix some $z \in \RR$ and recall the notation $u_N^\Lambda(z) = \tfrac{g(0)}{\alpha_\Lambda} \big\{ \log |\Lambda_N| + \log\log |\Lambda_N| + z \big\}$. We are interested in the punctual measure of the points of $\Lambda_N$ that are not covered by $\mathscr{I}_{\Psi_N}(u_N^\Lambda(z))$, that is
	\begin{equation}
		\mathcal{N}_N^{\Lambda,z} \defeq \sum_{x \in \Lambda_N} \delta_{x/N} \indic{x \not\in \mathscr{I}_{\psi_N}(u_N^\Lambda(z))} = \sum_{x \in \Lambda_N} \delta_{x/N}  \indic{U_x > u_N^\Lambda(z)} \, .
	\end{equation}

	\begin{lemma}\label{lem:def-mesure-surface}
		For $B \subseteq \Lambda$, we write $B_N = (N \cdot B) \cap \ZZ^d$ with $N \geq 1$. Then,
		\begin{equation}
			\lim_{N \to +\infty} \EE \big[ |\mathcal{N}_N^{\Lambda,0}(B_N)| \big] = \int_{B \cap (\partial \Lambda \cap \mathcal{L}_{\alpha_\Lambda})} \frac{\dd x}{\big| \nabla \varphi^2 (x) \big|} \eqdef \mu_\Lambda(B \cap (\partial \Lambda \cap \mathcal{L}_{\alpha_\Lambda})) \, .
		\end{equation}
		Note that this characterizes a finite measure on $\partial \Lambda \cap \mathcal{L}_{\alpha_\Lambda}$ that we also denote by $\mu_\Lambda$.
	\end{lemma}

	\begin{proof}
		The proof follows from the arguments in Appendix \ref{app:cv-somme-exp} applied to $B \cap \Lambda$ instead of $\Lambda$.
	\end{proof}
	
	We can now properly state the convergence of the ``late points'' towards a Poisson point process.
	
	\begin{theorem}
		Fix $z \in \RR$. The punctual measures $\mathcal{N}_N^{\Lambda,z}$ converge in distribution for the topology of weak convergence towards a Poisson point process $\mathcal{N}^{\Lambda,z}$ with intensity measure $e^{-z} \mu_\Lambda$.
	\end{theorem}

	\begin{proof}
	The proof uses a theorem by Kallenberg (see \cite[Proposition 3.22]{resnickExtremeValuesRegular1987}) stating that one only needs to check that
	\begin{equation}
		\lim_{N \to +\infty} \EE \big[ |\mathcal{N}_N^{\Lambda,z}(B_N)| \big] = e^{-z} \mu_\Lambda(B) \quad , \quad \lim_{N \to +\infty} \PP \big( \mathcal{N}_N^{\Lambda,z}(B_N) = 0 \big) = \exp \big(- \mu_\Lambda(B) e^{-z} \big) \, ,
	\end{equation}
	for any $B$ a closed ball of $\RR^d$ that intersects with $\Lambda$.
	
	The first equality is easily proved using the fact that
	\begin{equation}
		\EE \big[ |\mathcal{N}_N^{\Lambda,z}(B_N)| \big] = \frac{\log |\Lambda_N|}{|\Lambda_N| e^z} \sum_{x \in \Lambda_N} \exp \Big( - \big(\frac{\varphi_N^2(x)}{\alpha_\Lambda} - 1 \big)\big[ \log |\Lambda_N| - \log\log |\Lambda_N| + z \big] \Big)
	\end{equation}
	and a use of Lemma \ref{lem:sum-exp-varphi}.
	For the second equality, we observe that 
	\begin{equation}
		\PP \big( \mathcal{N}_N^{\Lambda,z}(B_N) = 0 \big) = \PP \big( B_N \subset \mathscr{I}_{\Psi_N}(u_N^\Lambda(z)) \big) \, .
	\end{equation}
	To get the limit as $N \to +\infty$, we will use the fact that the results of Section \ref{sec:caracterisation-Lambda-rho} still hold by replacing $\Lambda_N$ with $B_N$ but still keeping the same $u_N^\Lambda(\rho)$. For $\rho > 0$, we define
	\begin{equation}
		B_N(\rho) \defeq \big\{ x \in B_N \, : \, x \not\in \mathscr{I}_{\Psi_N}(u_N^{\Lambda,\rho}) \big\} \, ,
	\end{equation}
	where we stress that we kept the same intermediary level $u_N^{\Lambda,\rho} = (1-\rho) g(0) \alpha_\Lambda^{-1} \log |\Lambda_N|$. Then, using the same proofs as in Section \ref{sec:caracterisation-Lambda-rho} -- or observing that $u_N^{\Lambda,\rho} \geq u_N^{B,\rho}$ implies $B_N(\rho) \subseteq \Lambda_N(\rho)$, with self-explanatory notation -- we easily get that with high probability,
	\begin{equation}
		B_N(\rho) \lesssim \frac{|\Lambda_N|^\rho}{\log |\Lambda_N|} \quad , \quad \inf_{x,y \in B_N(\rho), x \neq y} |x-y| \geq |\Lambda_N|^{\frac{4 \rho}{d-2}} \, .
	\end{equation}
	Following the proof of Theorem \ref{th:gumbel+poisson-RI}, the probability $\PP \big( B_N \subset \mathscr{I}_{\Psi_N}(u_N^\Lambda(z)) \, | \, B_N(\rho) \big)$ is well approximated by
	\begin{equation*}
		\prod_{x \in B_N(\rho)} \PP \big( x \in \mathscr{I}_{\Psi_N}(u_N^\Lambda(z)-u_N^{\Lambda,\rho}) \big) \sim \exp \Big( - \sum_{x \in B_N(\rho)} \exp \big( - \tfrac{\varphi_N^2(x)}{\alpha_\Lambda} [\rho \log |\Lambda_N| - \log \log |\Lambda_N| + z] \big) \Big) \, .
	\end{equation*}
	We then conclude by observing that the sum in the exponential just above is
	\begin{equation}
		e^{-z} \frac{\log |\Lambda_N|}{|\Lambda_N|^\rho} \sum_{x \in B_N(\rho)} \exp \Big( - \rho \big[\tfrac{\varphi_N^2(x)}{\alpha_\Lambda} - 1 \big] \log |\Lambda_N|(1+\bar{o}(1)) \Big)
	\end{equation}
	and doing the proof of Proposition \ref{prop:convergence-Z_N-rho} to see that this converges to $e^{-z} \mu_\Lambda(B)$.
	\end{proof}

	\section{From tilted interlacements to the confined walk}\label{sec:transfert-RI-RW}

	We now explain how to transfer our result on the tilted interlacements to the confined walk. We first state the coupling theorem that we use.
	
	Denote by $\mathcal{R}_{\phi_N}(t_N)$ the range up to time $t_N \geq 1$ of the Markov chain with law $\Pbf^N_{\phi_N^2}$ (recall that $\Pbf^N_{\phi_N^2}$ is the tilted RW starting from its invariant measure $\mu=c_N \phi_N^2$).
	
	\begin{theorem}[Coupling theorem]\label{th:couplage-CRW-entrelac-t_N}
		Let $\delta \in (0,1)$ and consider a sequence $(t_N)_{N\geq 1}$ that satisfies $ t_N/ N^{2+\delta} \to +\infty$. We define
		\[
		u_N \defeq \frac{t_N }{N^d} \qquad \text{as well as} \qquad \eps_N \defeq N^{-\frac{\delta}{4}} \,. \]
		Then, there are some $\eta>0$ and some constants $c_1,c_2>0$ (that only depend on $\alpha, \delta, \eps, D$ and $d \geq 3$) and a coupling $\QQ_N$ of $\mathcal{R}_{\phi_N}(t_N)$ and $\mathscr{I}_{\Psi_N}((1 \pm \eps_N)u_N)$ such that, for all~$N$ large enough,
		\begin{equation}
			\mathscr{I}_{\Psi_N}((1-\eps_N)u_N) \cap B_N \, \subseteq\, \mathcal{R}_{\phi_N}(t_N) \cap B_N \,\subseteq \, \mathscr{I}_{\Psi_N}((1+\eps_N)u_N) \cap B_N \,,
		\end{equation}
		with $\QQ_N$-probability at least $ 1 - c_1 e^{-c_2 N^{\eta}}$.
	\end{theorem}

	We can now finally prove our main theorems by using the fact that $\eps_N$ in Theorem \ref{th:couplage-CRW-entrelac-t_N} can be neglected in our calculations.

	\begin{proof}[Proof of Theorems \ref{th:asymptotics-cover-RW} \& \ref{th:gumbel+poisson-RW}]
		Using Theorem \ref{th:couplage-CRW-entrelac-t_N}, we see that under the coupling, with probability at least $1 - c_1 e^{-c_2 N^\eta}$ we have
		\begin{equation}
			\mathcal{N}_{N}^{\Lambda, z_N^+} \subseteq \mathcal{N}_{N, \mathrm{RW}}^{\Lambda, z} \subseteq \mathcal{N}_{N}^{\Lambda, z_N^-} \quad , \quad \text{with} \quad z_N^\pm = z \pm \frac{g(0)}{\alpha_\Lambda} \eps_N \big\{ \log |\Lambda_N| - \log \log |\Lambda_N| \big\} \eqdef z \pm \eps_N' \, .
		\end{equation}
		In particular, since $\eps_N'$ decays polynomially fast, we can apply the proof of Lemma \ref{lem:def-mesure-surface}. Therefore, with the same notations, we get
		\begin{equation}
			\limsup_{N \to +\infty} \Ebf_{\phi_N^2} \big[ |\mathcal{N}_{N, \mathrm{RW}}^{\Lambda, z}(B_N)| \big] \leq \limsup_{N \to +\infty} \EE \big[ |\mathcal{N}_N^{\Lambda, z - \eps'_N}(B_N)| \big] = \mu_\Lambda(B) e^{-z} \, ,
		\end{equation}
		as well as
		\begin{equation}
			\liminf_{N \to +\infty} \Ebf_{\phi_N^2} \big[ |\mathcal{N}_{N, \mathrm{RW}}^{\Lambda, z}(B_N)| \big] \geq \liminf_{N \to +\infty} \EE \big[ |\mathcal{N}_N^{\Lambda, z + \eps'_N}(B_N)| \big] = \mu_\Lambda(B) e^{-z} \, .
		\end{equation}
		Similarly, we get the convergence of $\Pbf_{\phi_N^2} \big( |\mathcal{N}_{N, \mathrm{RW}}^{\Lambda, z \pm \eps_N'}(B_N)| = 0 \big)$ towards $\exp \big(- \mu_\Lambda(B) e^{-z} \big)$.
	\end{proof}

	\section{Some specific cases}
	
	\subsection{The case of a ball}\label{sec:application-boule}

	In this section, we will focus on the case where $D$ is the unit ball $B = B(0,1)$ and $\Lambda$ is the slighly smaller ball $B(0,1-\eps)$ for some $\eps \in (0,1)$. This is a much simpler case to consider, since the function $\varphi$ is radial: in particular $\mathcal{L}^\Lambda_{\alpha_\Lambda}$ is exactly the boundary $\partial \Lambda$.
	
	Let us first recall some facts about the eigenfunction $\varphi$ on a ball in $\RR^d$. Using the spherical coordinates $(r,\theta) \in \RR_+^* \times \mathbb{S}^{d-1}$, the Dirichlet problem \eqref{eq:Dirichlet} becomes
	\begin{equation}\label{eq:dirichlet-coord-spherique}
		0 = \lambda u(x) + \Delta u (x) = \lambda u(r,\theta) + \Big[ \frac{\partial^2}{\partial r^2} + \frac{d-1}{r} \frac{\partial}{\partial r} + \frac{1}{r^2} \Delta_{\mathbb{S}^{d-1}} \Big] u(r,\theta) \, ,
	\end{equation}
	where $\Delta_{\mathbb{S}^{d-1}}$ is the spherical Laplacian on $\RR^d$. Using the separation of variable, we write $u(r,\theta) = u_1(r) u_2(\theta)$ and thus \eqref{eq:dirichlet-coord-spherique} can be rewritten as
	\begin{equation}
		0 = r^2 u_1''(r) + r(d-1) u_1'(r) + \Big[ \lambda r^2 + \frac{\Delta_{\mathbb{S}^{d-1}} u_2(\theta)}{u_2(\theta)} \Big] u_1(r) \, .
	\end{equation}
	In particular, there is a $n \in \RR$ such that
	\begin{equation}
		r^2 u_1''(r) + r(d-1) u_1'(r) + \big[ \lambda r^2 - n \big] u_1(r) = 0 \quad , \quad \frac{\Delta_{\mathbb{S}^{d-1}} u_2(\theta)}{u_2(\theta)} = -n
	\end{equation}
	In the end, the first eigenfunction $\varphi$ on the unit ball can be expressed as
	\begin{equation}
		\varphi(r,\theta) = r^{1 - \frac{d}{2}} J_{\frac{d}{2}-1} (\sqrt{\lambda} r) \, ,
	\end{equation}
	with $J_\alpha$ the first-kind Bessel function of index $\alpha \in \RR$.
	Note that the properties of Bessel functions (see \cite[(10.6.6)]{NIST:DLMF}) imply $\tfrac{\dd}{\dd x} \big[ x^{-\alpha} J_{\alpha}(x) \big] = - x^{-\alpha} J_{\alpha + 1}(x)$. One can then deduce that the radial derivative of $\varphi^2$ is given by
	\begin{equation}\label{eq:derivee-varphi-boule}
		\frac{\partial}{\partial r} \varphi^2(r,\theta) = - 2 \sqrt{\lambda} r^{2-d} J_{1 - \frac{d}{2}}(\sqrt{\lambda} r) J_{2 - \frac{d}{2}}(\sqrt{\lambda} r) < 0 \, .
	\end{equation}
	
	We deduce in particular that the radial part of $\varphi^2$, which we denote by $\varphi_{2,R}$ is a smooth diffeomorphism, hence by change of variable we can compute the volume of the ``level bands'' of $\varphi^2$, and thus $\kappa_\Lambda$. Let us write $\psi_{2,R} \defeq (\varphi_{2,R})^{-1}$ the inverse of $\varphi_{2,R}$.

	\begin{proposition}
		Consider $D = B(0,1)$ and fix $r_0 \in (0,1)$. Then, if $\Lambda = B(0,r_0)$, we have
		\begin{equation}
			\kappa_\Lambda = \lim_{\eps \to 0} \tfrac{1}{\eps} \big| \big\{ x \in \Lambda \, : \, \varphi^2(x) \leq (1+\eps)\alpha_\Lambda \big\} \big| = \frac{r_0^{d-1}}{\big|\varphi_{2,R}'(r_0) \big|} \mathrm{Vol}_{d} \big(\mathbb{B}^{d}\big) \, ,
		\end{equation}
		with $\mathrm{Vol}_{d} \big(\mathbb{B}^{d}\big)$ the volume of the unit ball in $\RR^d$ and where $\varphi_{2,R}$ is the radial part of $\varphi^2$.
	\end{proposition}
	
	Note that this is consistent with the definition of $\kappa_\Lambda$ in Theorem \ref{th:gumbel+poisson-RW}.
	
	\begin{proof}
		Since the level sets of $\varphi^2$ are spheres of given radii, the set $\big| \big\{ x \in \Lambda \, : \, \varphi^2(x) \leq (1+\eps)\alpha_\Lambda \big\}$ is simply the ``annulus'' $\{ \psi_{2,R}(\alpha_\Lambda(1+\eps)) \leq |x| \leq \psi_{2,R}(\alpha_\Lambda) \}$. In particular,
		\begin{equation}
			\big| \big\{ x \in \Lambda \, : \, \varphi^2(x) \leq (1+\eps)\alpha_\Lambda \big\} \big| = \big[ (\psi_{2,R}(\alpha_\Lambda))^d - (\psi_{2,R}(\alpha_\Lambda(1+\eps)))^d \big]\mathrm{Vol}_{d} \big(\mathbb{B}^{d}\big) \, ,
		\end{equation}
		with $\mathrm{Vol}_{d} \big(\mathbb{B}^{d}\big)$ the volume of the unit ball in $\RR^d$. Dividing by $\eps > 0$ and having $\eps \to 0$ makes the derivative of $(\psi_{2,R})^d$ at $\alpha_\Lambda$ appear.
		
		Another way to compute $\kappa_\Lambda$ is through the following change of variable:
		\begin{equation}
		\begin{split}
			\big| \big\{ x \in \Lambda \, : \, \varphi^2(x) \leq (1+\eps)\alpha_\Lambda \big\} \big| &= \mathrm{Vol}_{d-1} \big(\mathbb{S}^{d-1}\big) \int_{\Lambda} \indic{\alpha_\Lambda \leq \varphi^2(r) \leq (1+\eps)\alpha_\Lambda} \, r^{d-1} \dd r\\
			&= \mathrm{Vol}_{d-1} \big(\mathbb{S}^{d-1}\big) \int_{\alpha_\Lambda}^{(1+\eps)\alpha_\Lambda} \big| J_{\psi_{2,R}}(y) \big| \, (\psi_{2,R}(y))^{d-1} \dd y \, ,
		\end{split}
		\end{equation}
		with $J_{\Psi_{2,R}} = (\psi_{2,R})' = [(\varphi_{2,R})' \circ \psi_{2,R}]^{-1}$ the Jacobian of $\psi_{2,R}$ and $\mathrm{Vol}_{d-1} \big(\mathbb{S}^{d-1}\big)$ the $(d-1)$-volume of the unit sphere in $\RR^d$. Dividing by $\eps > 0$ and having $\eps \to 0$, this converges to $\big| J_{\Psi_{2,R}}(\alpha_\Lambda) \big| = [(\varphi_{2,R})' \circ \psi_{2,R}(\varphi_{2,R}(r_0))]^{-1}$.
	\end{proof}

	\subsection{Covering of a segment by the conditioned simple random walk}
	
	In this section, we will deviate from the confined walk to instead consider the conditioned walk, that is the simple random walk conditioned to stay in $D_N$ up to some time $T_N$.
	In the following, we consider the segment $I_N \defeq \llbracket -N , N \rrbracket \subset \ZZ$, where we write $\llbracket a , b \rrbracket \defeq [a,b] \cap \ZZ$. We also consider the time horizon $T_N$ to be far greater than $N^3$.
	
	Let $\mathfrak{C}_N$ be the covering time of $I_N$ by the simple random walk conditioned to stay in $I_N$, then we can conjecture the convergence in distribution
	\begin{equation}
		\frac{\mathfrak{C}_N}{N^3} \xrightarrow[N \to +\infty]{(d)} \mathfrak{C} \sim \sum_{i = 1}^{G} \xi_i \quad , \quad \text{with $G \sim \mathcal{G}(\tfrac12)$ and $(\xi_i)$ are i.i.d. variables} \, ,
	\end{equation}
	
	Let us give some heuristics for this convergence. We write $\tau^0_0 = 0$ and define for $i \geq 1$:
	\begin{equation}
		\tau_i \defeq \inf \big\{ t \geq \tau^0_{i-1} \, : \, |S_t| = N \big\} \quad , \quad \eps_i = \mathrm{sign}(S_{\tau_i}) \quad , \quad \tau^0_{i} = \inf \big\{ t \geq \tau_i \, : \, S_t = 0 \big\}
	\end{equation}
	Observe that $\mathfrak{C}_N = \tau_{\mathcal{I}}$, where $\mathcal{I}$ is the first $i$ such that $\eps_i \neq \eps_1$ (that is the walk reached both of the extremities of the segment). With precise gambler's ruin formulae, one can prove that $\tau_i - \tau^0_{i-1} \asymp N^3$ with high probability (see below). On the other hand, the return times to $0$ are asymptotically of order $N^2$, as the drift facilitates the return to zero, hence negligible. Therefore, $\mathfrak{C}_N \approx \sum_{i \leq \mathcal{I}} (\tau_i - \tau^0_{i-1})$ with the $\tau_i - \tau^0_{i-1}$ scaling to the first time a Brownian motion conditioned to stay in $[-1,1]$ reaches the boundary.
	
	The main point here is the fact that the order is understandable from the point of view of Theorem \ref{th:asymptotics-cover-RW}. Indeed, it is known that the first eigenfunction of the Laplace operator on the segment is of order $1/N$ at the boundary, hence the $\alpha_\Lambda$ in this case would be $\asymp N^{-2}$ and the asymptotics is $N^d/\alpha_\Lambda \asymp N^3$. The disappearance of the factor $\log N$ is explained by the fact that reaching the boundary implies covering the full half-segment, thus at no time there exists a set of scattered points that the walk collects independently.

	\begin{appendix}
		
		\section{Appendix: technical estimates and decoupling inequalities}
		
		\subsection{Asymptotics of the Green function of the tilted walk}\label{app:green-function}
		
		\begin{proof}[Proof of Proposition \ref{prop:encadrement-green-tilté}]
			Fix $\eta > 0$. For $x \in B_N$ and $R \in [\tfrac{1}{\eta}, \eta N]$, we consider the annulus intersected with $B_N$, which we write $A_N^x(R) \defeq (B(x,2R) \setminus B(x,R)) \cap B_N$. The main part of the proof is to prove that there exists constants $c,c' > 0$, that neither depends on $N$ large enough, nor on $x$ and $R$, such that
			\begin{equation}\label{claim:green-tilt}
				\text{for all } y \in A_N^x(R) , \quad c R^{2-d} \leq G^{\Psi_N}(x,y) \leq c' R^{2-d} \, .
			\end{equation}
			Afterwards, we prove that $y \mapsto G^{\Psi_N}(x,y)$ is non-increasing in $|x-y|$, which suffices to get the upper bound; while the lower bound will be easily deduced from the proof of the main point \eqref{claim:green-tilt}.
			
			We first claim that there exists a constant $\kappa' > 0$, that neither depends on $N$ large enough nor on $x$ and $R$, that is such that 
			\begin{equation}\label{eq:harnack-fct-green}
				\forall y,z \in A_N^x(R) \, , \quad \kappa' G^{\Psi_N}(x,y) \leq G^{\Psi_N}(x,z) \leq \tfrac{1}{\kappa'} G^{\Psi_N}(x,y) \, .
			\end{equation}
			From \eqref{eq:harnack-fct-green}, we can deduce that
			\begin{equation}\label{eq:fct-green-moyenne}
				\kappa' |A_N^x(R)| G^{\Psi_N}(x,y) \leq \sum_{z \in A_N^x(R)} G^{\Psi_N}(x,z) \leq \tfrac{1}{\kappa'} |A_N^x(R)| G^{\Psi_N}(x,y) \, .
			\end{equation}
			Our second claim is that there exists a constant $\kappa'' > 0$ that neither depends on $N$ large enough, nor on $x$ and $R$ such that 
			\begin{equation}\label{eq:somme-fct-green-anneau}
				\kappa'' R^2 \leq \sum_{z \in A_N^x(R)} G^{\Psi_N}(x,z) = \Ebf^{\Psi_N}_x \Big[ \sum_{k \geq 0} \indic{X_k \in A_N^x(R)} \Big] \leq \tfrac{1}{\kappa''} R^2 \, .
			\end{equation}
			Combining \eqref{eq:somme-fct-green-anneau} with \eqref{eq:fct-green-moyenne} yields \eqref{claim:green-tilt}.
			
			\textit{Proof of \eqref{eq:harnack-fct-green}}: Fix $\delta > 0$ and $y \in A_N^x(R)$ so that $|x-y| < 2 \delta R$, and consider $w \in B(y,\tfrac14 \delta R)$. Then, by the reversibility of $\Psi_N^2$ and the Markov property, we get
			\begin{equation}
				\frac{\Psi_N^2(x)}{\Psi_N^2(w)} G^{\Psi_N}(x,w) = G^{\Psi_N}(w,x) = \sum_{u \in \partial B(y,\delta R)} \Pbf^{\Psi_N}_w(X_{H_{B(y,\delta R)}} = u) G^{\Psi_N}(u,x) \, .
			\end{equation}
			
			Now, we claim that there is a constant $\kappa_y > 0$, uniform in $R$ large enough, $\delta, \eta$ small enough and $w \in B(y,\tfrac14 \delta R)$ such that $R^{d-1} \Pbf^{\Psi_N}_w(X_{H_{B(y,\delta R)}} = u) \in [\kappa_y, \kappa_y^{-1}]$. In particular, since the ratio of $\Psi_N^2$ is also bounded by constants uniform in $x,y,w$, for all $w \in B(y,\tfrac14 \delta R)$, we have $G^{\Psi_N}(x,w)/G^{\Psi_N}(x,y) \in [\kappa_y', (\kappa_y')^{-1}]$. The proof can be found in \cite[Lemma A.2.3]{these} Since $A_x(R)$ can be covered by a finite number (uniform in $N$ and $R$ large enough, and in $x \in B_N$) of balls with radius $\tfrac14 \delta R$, we get \eqref{eq:harnack-fct-green}.
			
			\textit{Proof of \eqref{eq:somme-fct-green-anneau}}:
			Fix $\delta > 0$ small enough and consider $A_{N,\delta}^x(R) = (A_N^x(R))^{\delta R} = B(x, (2+\delta)R) \setminus B(x,(1-\delta) R)$. We easily see that the time spent in $A_N^x(R)$ is less than the sum of length of the excursions $A_N^x(R) \to \partial A_{N,\delta}^x(R)$, which is what we will use to get the bound. Write $M$ for the number of such excursions, we get the bounds
			\begin{equation}
				\Ebf^{\Psi_N}_x \Big[ \sum_{k \geq 0} \indic{X_k \in A_N^x(R)} \Big] \leq \Ebf^{\Psi_N}_x[M] \sup_{w \in A_N^x(R)} \Ebf^{\Psi_N}_w \Big[ H_{\partial A_{N,\delta}^x(R)} \Big] \, .
			\end{equation}
			Using \cite[Lemma 2.5,2.6]{bouchot2024confinedrandomwalklocally}, $M$ is dominated by a geometric random variable with parameter at least some $p > 0$ independent from $R$ large enough. In particular $\Ebf^{\Psi_N}_x[G]$ is bounded from above by a constant $c > 0$ that does not depend on $R$ large enough. On the other hand, we can write using \eqref{eq:encadrement-proba}:
			\begin{equation}
				\Ebf^{\Psi_N}_w \Big[ H_{\partial A_{N,\delta}^x(R)} \Big] \leq \kappa R^2 \sum_{k \geq 0} (k+1) e^{c_0 k (R/N)^2} \Pbf_w( H_{\partial A_{N,\delta}^x(R)} \geq k R^2) \, .
			\end{equation}
			It is known that $\Pbf_w( H_{\partial A_{N,\delta}^x(R)} \geq k R^2) \leq c e^{-k/\delta^2}$. Therefore, provided $R \leq \eta N$ with $\eta > 0$ small enough, there is a constant $c_\eta > 0$ that does not depend on $R$ such that
			\begin{equation}
				\Ebf^{\Psi_N}_w \Big[ H_{\partial A_{N,\delta}^x(R)} \Big] \leq c R^2 \sum_{k \geq 0} (k+1) e^{-c_\eta k} \leq \tfrac{1}{\kappa''} R^2 \, .
			\end{equation}
			This proves the upper bound in \ref{eq:somme-fct-green-anneau}. To get the lower bound, we use again \eqref{eq:encadrement-proba} and conclude using $\Ebf_w \big[ H_{\partial A_{N,\delta}^x(R)} \big] \geq c R^2$ (an easy SRW estimate), hence proving \eqref{eq:somme-fct-green-anneau}.
			
			\textit{Farther points}: if $|x-y| > \eta N$, we prove that $G^{\Psi_N}(x,y)$ can be controlled by the Green function between $x$ and a point at distance $\tfrac12 \eta N$ from $x$. With the reversibility of $\Psi_N^2$ and the Markov property, we easily get
			\begin{equation}
				G^{\Psi_N}(x,y) = \frac{\Psi_N^2(y)}{\Psi_N^2(x)} G^{\Psi_N}(y,x) = \frac{\Psi_N^2(y)}{\Psi_N^2(x)} \sum_{z \in \partial B(x,\frac12 \eta N)} G^{\Psi_N}(z,x) \Pbf^{\Psi_N}_y(X_{H_{B(x,\frac12 \eta N)}} = z) \, .
			\end{equation}
			Applying \eqref{claim:green-tilt} and \eqref{eq:encadrement-ratio}, we get that $G^{\Psi_N}(x,y) \asymp N^{2-d} \Pbf^{\Psi_N}_y(H_{B(x,\frac12 \eta N)} < +\infty)$. We can show using \eqref{eq:encadrement-proba} and gambler's ruin-type results about the SRW that $\Pbf^{\Psi_N}_y(H_{B(x,\frac12 \eta N)} < +\infty) \geq c_1$ for some constant $c_1 > 0$ that neither depends on $N$ nor on $y \in \Lambda_N$, thus ending the proof of the Proposition.
		\end{proof}

		\subsection{Proof of Lemma \ref{lem:sum-exp-varphi}}\label{app:cv-somme-exp}
		
		We introduce the notation
		\begin{equation}
			\mathcal{L}^\Lambda(\alpha_1, \alpha_2) \defeq \big\{ x \in \Lambda \, : \, \varphi^2(x) \in [\alpha_1, \alpha_2) \big\} \, .
		\end{equation}
		The main ingredient of the proof of Lemma \ref{lem:sum-exp-varphi} is the following result, which critically requires Assumption \ref{hyp:deformation-level-sets} to hold.
		
		\begin{proposition}\label{prop:cardinal-level-sets}
		Under Assumption \ref{hyp:deformation-level-sets}, we have the convergence
			\begin{equation}
				\lim_{\eps \to 0} \frac{1}{\eps} \big| \mathcal{L}^\Lambda(\alpha_\Lambda, (1+\eps)\alpha_\Lambda) \big| = \int_{\Lambda \cap \mathcal{L}_{\alpha_\Lambda}} \frac{\dd x}{\big| \nabla \varphi^2 (x) \big|} \eqdef \kappa_\Lambda \, .
			\end{equation}
		\end{proposition}

		Let us first prove Lemma \ref{lem:sum-exp-varphi} using Proposition \ref{prop:cardinal-level-sets}. We recall that we must control the sum
		\begin{equation}
			\frac{\log |\Lambda_N|}{|\Lambda_N|^{1-\beta}} \sum_{x \in \Lambda_N} \exp \Big( - \beta \Big[ \frac{\varphi_N^2(x)}{\alpha_\Lambda} - 1 \Big] \log |\Lambda_N| \Big) = \frac{\log |\Lambda_N|}{|\Lambda_N|} \sum_{x \in \Lambda_N} \exp \Big( - \beta \Big[ \frac{\varphi_N^2(x)}{\alpha_\Lambda} - 1 \Big] \log |\Lambda_N| \Big) \, .
		\end{equation}
		The proof of Lemma \ref{lem:sum-exp-varphi} relies on spliting $\mathcal{L}(\alpha_\Lambda, \alpha_\Lambda(1+\eps_0))$ into ``level bands'' and using Proposition \ref{prop:cardinal-level-sets} to control the number of terms in these level bands.
		
		Fix $\eta > 0$ then by Proposition \ref{prop:cardinal-level-sets}, there exists $\eps_0 > 0$ such that for all $\eps \in (0,\eps_0)$ we have
		\begin{equation}\label{eq:encadrement-lavel-band-eta}
			(1-\eta) \eps \kappa_\Lambda \leq |\mathcal{L}^\Lambda(\alpha_\Lambda, \alpha_\Lambda(1+\eps_0))| \leq (1+\eta) \eps \kappa_\Lambda \, .
		\end{equation}
		Turning back to the sum of Lemma \ref{lem:sum-exp-varphi}, we first split the sum depending on whether $\varphi_N^2(x) \leq \alpha_\Lambda(1+\eps_0)$ or not.
		
		For $x$'s such that $\varphi_N^2(x) > \alpha_\Lambda(1+\eps_0)$, we have
		\begin{equation}\label{eq:sum-exp-points-loins}
			\sum_{\substack{x \in \Lambda_N\\ \varphi_N^2(x) > \alpha_\Lambda(1+\eps_0)}} \exp \Big( - \beta \tfrac{\varphi_N^2(x)}{\alpha_\Lambda} \log |\Lambda_N| \Big) \leq |\Lambda_N| e^{- \beta (1+\eps_0) \log |\Lambda_N|} = |\Lambda_N|^{1-\beta} |\Lambda_N|^{-\beta \eps_0} \, .
		\end{equation}
		After multiplying by $|\Lambda_N|^{\beta-1} \log |\Lambda_N|$, this vanishes at $N \to +\infty$.
		
		Let us now consider the $x$'s such that $\varphi_N^2(x) \leq (1+\eps_0)\alpha_\Lambda$. Fix $\delta > 0$. For $k \geq 0$ we write
		\begin{equation}
			\alpha_k \defeq \alpha_\Lambda \Big( 1 + \frac{k \delta \eps_0}{\log |\Lambda_N|} \Big) \, .
		\end{equation}
		Then, writing $k(\eps_0) = \tfrac{\eps_0}{\delta} \log |\Lambda_N| - 1$, we have
		\begin{equation*}
			\sum_{\substack{x \in \Lambda_N\\ \varphi_N^2(x) \leq \alpha_\Lambda(1+\eps_0)}} \exp \Big( - \beta \tfrac{\varphi_N^2(x)}{\alpha_\Lambda} \log |\Lambda_N| \Big) = \sum_{k = 1}^{k(\eps_0)} \sum_{x \in \mathcal{L}^\Lambda(\alpha_k, \alpha_{k+1})} \exp \Big( - \beta \tfrac{\varphi_N^2(x)}{\alpha_\Lambda} \log |\Lambda_N| \Big)
		\end{equation*}
		By the definition of $\alpha_k$, we have
		\begin{equation*}
			|\mathcal{L}^\Lambda(\alpha_k, \alpha_{k+1})| e^{-\beta \delta \eps_0 (k+1)} \leq \sum_{x \in \mathcal{L}^\Lambda(\alpha_k, \alpha_{k+1})} \exp \Big( - \beta \tfrac{\varphi_N^2(x)}{\alpha_\Lambda} \log |\Lambda_N| \Big) \leq e^{-\beta \delta \eps_0 k} |\mathcal{L}^\Lambda(\alpha_k, \alpha_{k+1})| \, .
		\end{equation*}
		Since $\delta > 0$ can be arbitrarily small, we are left to study
		\begin{equation}\label{eq:sum-proche-min-bandes}
			|\Lambda_N|^{-\beta} \sum_{k = 1}^{k(\eps_0)} e^{- \beta \delta \eps_0 k} |\mathcal{L}^\Lambda(\alpha_k, \alpha_{k+1})| \, .
		\end{equation}
		
		Let us now do a discrete integration by part by writing
		\begin{equation*}
			e^{- \beta \delta \eps_0 k} = e^{-\beta \delta \eps_0 [k(\eps_0)+1]} - \sum_{j = k}^{k(\eps_0)} \big( e^{-\beta \delta \eps_0 (j+1)} - e^{- \beta \delta \eps_0 j} \big) \, .
		\end{equation*}
		The sum in \eqref{eq:sum-proche-min-bandes} can thus be rewritten as
		\begin{equation}
			|\mathcal{L}^\Lambda(\alpha_\Lambda, \alpha_{k(\eps_0)})| e^{-\beta \delta \eps_0 [k(\eps_0)+1]} + \sum_{k = 1}^{k(\eps_0)} \sum_{j = k}^{k(\eps_0)} \big( e^{-\beta \delta \eps_0 j} - e^{- \beta \delta \eps_0 (j+1)} \big) |\mathcal{L}^\Lambda(\alpha_k, \alpha_{k+1})|
		\end{equation}
		Interverting the sums on $k$ and $j$ and noticing that $\sum_{k = 0}^{j} |\mathcal{L}^\Lambda(\alpha_k, \alpha_{k+1})| = |\mathcal{L}^\Lambda(\alpha_\Lambda, \alpha_{j+1})|$, we are left with
		\begin{equation}\label{eq:apres-IPPD}
			|\mathcal{L}^\Lambda(\alpha_\Lambda, \alpha_{k(\eps_0)})| e^{-\beta \delta \eps_0 [k(\eps_0)+1]} + \big(1 - e^{-\beta \delta \eps_0} \big) \sum_{j = 0}^{k(\eps_0)} e^{- \beta \delta \eps_0 j} |\mathcal{L}^\Lambda(\alpha_\Lambda, \alpha_{j+1})|
		\end{equation}
		Using the definition of $k(\eps_0)$ on the first term of \eqref{eq:apres-IPPD}, we get
		\begin{equation*}
			|\mathcal{L}^\Lambda(\alpha_\Lambda, \alpha_{k(\eps_0)})| e^{-\beta \delta \eps_0 [k(\eps_0)+1]} = |\mathcal{L}^\Lambda(\alpha_\Lambda, \alpha_{k(\eps_0)})| |\Lambda_N|^{-\beta \eps_0^2} \leq |\Lambda_N|^{1-\beta \eps_0^2} \, .
		\end{equation*}
		For the second term of \eqref{eq:apres-IPPD}, we use \eqref{eq:encadrement-lavel-band-eta} to get
		\begin{equation*}
			(1-\eta) \frac{N^d \delta \eps_0}{\log |\Lambda_N|}\sum_{j = 0}^{k(\eps_0)} (j+1) e^{- \beta \delta \eps_0 j} \leq \sum_{j = 0}^{k(\eps_0)} e^{- \beta \delta \eps_0 j} |\mathcal{L}^\Lambda(\alpha_\Lambda, \alpha_{j+1})| \leq (1+\eta) \frac{N^d \delta \eps_0}{\log |\Lambda_N|}\sum_{j = 0}^{k(\eps_0)} (j+1) e^{- \beta \delta \eps_0 j} \, .
		\end{equation*}
		We easily conclude by using the fact that
		\begin{equation}
			\big(1 - e^{-\beta \delta \eps_0} \big) \delta \eps_0 \sum_{j = 0}^{k(\eps_0)} (j+1) e^{- \beta \delta \eps_0 j} = \big(1 - e^{-\beta \delta \eps_0} \big) \delta \eps_0 \frac{e^{-\beta \delta \eps_0} \big( 3 e^{\beta \delta \eps_0} - 2\big)}{(\big( e^{\beta \delta \eps_0} - 1 \big)^2} \overset{\delta \downarrow 0}{\sim} \frac{\beta (\delta \eps_0)^2}{(\beta \delta \eps_0)^2} = \frac{1}{\beta} \, .
		\end{equation}
		
		\begin{proof}[Proof of Proposition \ref{prop:cardinal-level-sets}]
			Consider $\eps_0 > 0$ that is given by Assumption \ref{hyp:deformation-level-sets}.	Note that $\Lambda$ is a compact manifold and that for all $\eps \in (0,\eps_0), \mathcal{L}(\alpha_\Lambda, (1+\eps)\alpha_\Lambda)$ is compact. Note that $\varphi^2$ is smooth on the interior of $D$ and has no critical points in $\mathcal{L}(\alpha_\Lambda, (1+\eps)\alpha_\Lambda)$. Therefore, according to the proof of \cite[Theorem 3.1]{milnorMorseTheory2016}, there exists a unique solution $(\Phi_\delta)_{\delta \in (0,\eps_0)}$ to the modified gradient flow equation
			\begin{equation}
				\Phi_0(x) = x \quad , \quad \frac{\dd}{\dd \delta} \Phi_\delta(x) = \frac{\nabla \varphi^2}{|\nabla \varphi^2|^2} \big[ \Phi_\delta(x) \big] \, ,
			\end{equation}
			which is such that $\Phi_\delta$ is a diffeomorphism $\mathcal{L}_{\alpha_\Lambda} \longrightarrow \mathcal{L}_{(1+\delta)\alpha_\Lambda}$ \, .
			We can then write
			\begin{equation}
				\big| \mathcal{L}^\Lambda(\alpha_\Lambda, (1+\eps)\alpha_\Lambda) \big| = \int_0^\eps \dd \delta \int_{\Phi_\delta(\Lambda \cap \mathcal{L}_{\alpha_\Lambda})} \dd x = \int_0^\eps \dd \delta \int_{\Lambda \cap \mathcal{L}_{\alpha_\Lambda}} \frac{\mathrm{det}(\dd \Phi_\delta(x))}{\big| \nabla \varphi^2 \circ \Phi_\delta(x) \big|} \, \dd x
			\end{equation}
			Let us first claim that the integral on $x$ is a continuous function of $\delta$, and thus is uniformly continuous. Therefore, we have
			\begin{equation}
				\lim_{\eps \to 0} \frac{1}{\eps} \big| \mathcal{L}^\Lambda(\alpha_\Lambda, (1+\eps)\alpha_\Lambda) \big| = \int_{\Lambda \cap \mathcal{L}_{\alpha_\Lambda}} \frac{\mathrm{det}(\dd \Phi_0(x))}{\big| \nabla \varphi^2 \circ \Phi_0(x) \big|} \, \dd x = \int_{\Lambda \cap \mathcal{L}_{\alpha_\Lambda}} \frac{\dd x}{\big| \nabla \varphi^2 (x) \big|} \, ,
			\end{equation}
			the last equality using the fact that $\Phi_0(x) = x$.
			The continuity of the integral with respect to $x$ follows from the fact that $\delta \mapsto \Phi_\delta$ is $\mathcal{C}^1$ meaning that the Jacobian is continuous. In the same way, the denominator of the integrand is continuous and bounded from above. Hence, the integrand is continuous on a compact set and Lebesgue continuity theorem gives us the claim.
		\end{proof}

	\subsection{Decoupling of interlacements}

	\begin{proposition}\label{prop:decouplage-recouvrement}
		Let $K \subseteq \Lambda_N$ and $u > 0$, and define $s(K) \defeq \inf_{x,y \in K, x \neq y} |x-y|^{d-2}$. Then, there is a constant $c = c(d,\Lambda) > 0$ such that for $N$ large enough we have the decoupling inequality
		\begin{equation}
			\Big| \PP \Big( K \subset \mathscr{I}_{\Psi_N}(u) \Big) - \prod_{x \in K} \PP \big(x \in \mathscr{I}_{\Psi_N}(u) \big) \Big| \leq c u \frac{|K|^2}{s(K)} \, .
		\end{equation}
	\end{proposition}

	\begin{remark}
		Several works in the literature give far better bounds for the decoupling of distant regions of random interlacements, see \text{e.g.} \cite{popovSoftLocalTimes2015}. However, they rely on some sprinkling of the intensity, which is inconvenient when working with a large number of points/sets to decouple. 
	\end{remark}
	
	The proof relies on the following lemma, that holds for any pair of disjointed finite sets. The result is fairly standard in the case of random interlacements (see \cite[Claim 8.1]{drewitz2014introduction} or \cite[Lemma 2.1]{beliusCoverLevelsRandom2012}): Lemma \ref{lem:decouplage-2-sets} simply states that it still holds for tilted interlacements. The proof is identical, only replacing the Newtonian capacity and the Green function by their tilted equivalent, and using Proposition \ref{prop:encadrement-green-tilté} to bound the Green function.
	
	\begin{lemma}\label{lem:decouplage-2-sets}
		Let $K_1,K_2 \subset \ZZ^d$ be such that $K_1 \cap K_2 = \varnothing$, and consider two events $A_1^u$ and $A_2^u$ that only depend respectively on $\mathscr{I}_{\Psi_N}(u) \cap K_1$ and $\mathscr{I}_{\Psi_N}(u) \cap K_2$. Then, there is a constant $C_d > 0$ that only depends on the dimension such that
		\begin{equation}
			\Big| \PP(A_1^u \cap A_2^u) - \PP(A_1^u)\PP(A_2^u) \Big| \leq C_d u \frac{\cpc^{\Psi_N}(K_1) \cpc^{\Psi_N}(K_2)}{d(K_1,K_2)^{d-2}} \, .
		\end{equation}
	\end{lemma}
	
	\begin{proof}[Proof of Proposition \ref{prop:decouplage-recouvrement}]
		Choose $x \in K$, we apply Lemma \ref{lem:decouplage-2-sets} to the sets $K_1 = \mathset{x}$ and $K_2 = K \setminus \mathset{x}$, and the events $A_1 = \mathset{x \in \mathscr{I}_{\Psi_N}(u)}$ and $A_2 = \mathset{K \setminus \mathset{x} \subset \mathscr{I}_{\Psi_N}(u)}$. This yields
		\begin{equation}
			\Big| \PP( K \subset \mathscr{I}_{\Psi_N}(u) ) - \PP(x \in \mathscr{I}_{\Psi_N}(u))\PP(K \setminus \mathset{x} \subset \mathscr{I}_{\Psi_N}(u)) \Big| \leq C_d u \frac{\cpc^{\Psi_N}(\{x\}) \cpc^{\Psi_N}(K \setminus \mathset{x})}{d(x,K \setminus \mathset{x})^{d-2}} \, .
		\end{equation}
		Note that since $x \in K \subseteq \Lambda_N$, by Proposition \ref{prop:cap-tilt=phi2} and $\varphi_N \leq \kappa_2^{-1}$, we have $\cpc^{\Psi_N}(\{ x \}) \leq \kappa_2^{-2} (1 + N^{-c})/g(0)$ as well as
		\begin{equation}
			\cpc^{\Psi_N}(K \setminus \mathset{x}) \leq \lambda_N \sum_{z \in K} \phi_N^2(z) \Pbf^{\Psi_N}_x (\bar{H}_{K \setminus \mathset{x}} = +\infty) \leq \sum_{z \in K} \phi_N^2(z) \, .
		\end{equation}
		Using \eqref{eq:phiN-borne} ($\phi_N$ is bounded), we deduce that $\cpc^{\Psi_N}(K \setminus \mathset{x}) \leq (1 + N^{-c})\kappa_2^{-2} |K|$ and thus
		\begin{equation}
			\Big| \PP( K \subset \mathscr{I}_{\Psi_N}(u) ) - \PP(x \in \mathscr{I}_{\Psi_N}(u))\PP(K \setminus \mathset{x} \subset \mathscr{I}_{\Psi_N}(u)) \Big| \leq C_d u  \frac{(1 + N^{-c})^2}{g(0) \kappa_2^4} u \frac{|K|}{s(K)} \, .
		\end{equation}
		Repeating the steps $|K|-1$ times and using the triangular inequality concludes the proof of the proposition with $c = 2 C_d / g(0) \kappa_2^4$.
	\end{proof}

	\end{appendix}

	\printbibliography[heading=bibintoc]

@misc{bouchot2024confinedrandomwalklocally,
      title={A confined random walk locally looks like tilted random interlacements}, 
      author={Nicolas Bouchot},
      year={2024},
      eprint={2405.14329},
      archivePrefix={arXiv},
      primaryClass={math.PR},
      url={https://arxiv.org/abs/2405.14329}, 
}

@misc{berger2025propertiesprincipaldirichleteigenfunction,
      title={Some properties of the principal Dirichlet eigenfunction via simple random walk and Brownian motion couplings}, 
      author={Quentin Berger and Nicolas Bouchot},
      year={2025},
      eprint={2408.15858},
      archivePrefix={arXiv},
      primaryClass={math.PR},
}

@book{drewitz2014introduction,
  title={An introduction to random interlacements},
  author={Drewitz, Alexander and R{\'a}th, Bal{\'a}zs and Sapozhnikov, Art{\"e}m},
  year={2014},
  publisher={Springer}
}

@article{sznitman2010vacant,
  title={Vacant set of random interlacements and percolation},
  author={Sznitman, Alain-Sol},
  journal={Annals of mathematics},
  pages={2039--2087},
  year={2010},
  publisher={JSTOR}
}

@book{lawlerRandomWalkModern2010,
  title = {Random {{Walk}}: {{A Modern Introduction}}},
  shorttitle = {Random {{Walk}}},
  author = {Lawler, Gregory F. and Limic, Vlada},
  year = {2010},
  month = jun,
  edition = {1},
  publisher = {{Cambridge University Press}},
  doi = {10.1017/CBO9780511750854},
  urldate = {2023-02-07},
  isbn = {978-0-521-51918-2 978-0-511-75085-4}
}

@article{windischRandomWalkDiscrete2008,
  title = {Random Walk on a Discrete Torus and Random Interlacements},
  author = {Windisch, David},
  year = {2008},
  month = jan,
  journal = {Electronic Communications in Probability},
  volume = {13},
  number = {none},
  pages = {140--150},
  publisher = {Institute of Mathematical Statistics and Bernoulli Society},
  issn = {1083-589X, 1083-589X},
  doi = {10.1214/ECP.v13-1359}
}

@article{liLowerBoundDisconnection2014,
  title = {A Lower Bound for Disconnection by Random Interlacements},
  author = {Li, Xinyi and Sznitman, Alain-Sol},
  year = {2014},
  month = jan,
  journal = {Electronic Journal of Probability},
  volume = {19},
  number = {none},
  issn = {1083-6489},
  doi = {10.1214/EJP.v19-3067}
}

@book{levin2017markov,
  title={Markov chains and mixing times},
  author={Levin, David A and Peres, Yuval and Wilmer, L. Wilmer},
  volume={107},
  year={2017},
  publisher={American Mathematical Soc.}
}

@article{pjm/1103040107,
author = {H. F. Weinberger},
title = {{Lower bounds for higher eigenvalues by finite difference methods.}},
volume = {8},
journal = {Pacific Journal of Mathematics},
number = {2},
publisher = {Pacific Journal of Mathematics, A Non-profit Corporation},
pages = {339 -- 368},
year = {1958},
}

@misc{prévost2023phase,
      title={Phase transition for the late points of random walk}, 
      author={Alexis Prévost and Pierre-François Rodriguez and Perla Sousi},
      year={2023},
      eprint={2309.03192},
      archivePrefix={arXiv},
      primaryClass={math.PR}
}

@article{aldousTimeTakenRandom1983,
  title = {On the Time Taken by Random Walks on Finite Groups to Visit Every State},
  author = {Aldous, David J.},
  year = {1983},
  journal = {Zeitschrift für Wahrscheinlichkeitstheorie und Verwandte Gebiete},
  volume = {62},
  number = {3},
  pages = {361--374},
  issn = {0044-3719, 1432-2064},
  doi = {10.1007/BF00535260},
  langid = {english}
}

@article{beliusGumbelFluctuationsCover2013,
  title = {Gumbel Fluctuations for Cover Times in the Discrete Torus},
  author = {Belius, David},
  year = {2013},
  month = dec,
  journal = {Probability Theory and Related Fields},
  volume = {157},
  number = {3},
  pages = {635--689},
  issn = {1432-2064},
  doi = {10.1007/s00440-012-0467-7}
}

@article{beliusCoverLevelsRandom2012,
  title = {Cover Levels and Random Interlacements},
  author = {Belius, David},
  year = {2012},
  month = apr,
  journal = {The Annals of Applied Probability},
  volume = {22},
  number = {2},
  issn = {1050-5164},
  doi = {10.1214/11-AAP770}
}

@article{teixeiraFragmentationTorusRandom2011,
  title = {On the Fragmentation of a Torus by Random Walk},
  author = {Teixeira, Augusto and Windisch, David},
  year = {2011},
  journal = {Communications on Pure and Applied Mathematics},
  volume = {64},
  number = {12},
  pages = {1599--1646},
  issn = {1097-0312},
  doi = {10.1002/cpa.20382},
  urldate = {2024-04-20},
  langid = {english}
}

@article{cernyRandomWalksTorus2016b,
  title = {Random Walks on Torus and Random Interlacements: {{Macroscopic}} Coupling and Phase Transition},
  shorttitle = {Random Walks on Torus and Random Interlacements},
  author = {{\v C}ern{\'y}, Ji{\v r}{\'i} and Teixeira, Augusto},
  year = {2016},
  month = oct,
  journal = {The Annals of Applied Probability},
  volume = {26},
  number = {5},
  pages = {2883--2914},
  publisher = {Institute of Mathematical Statistics},
  issn = {1050-5164, 2168-8737},
  doi = {10.1214/15-AAP1165},
}

@phdthesis{these,
	title = {Localisation de polymères en milieux aléatoires :
obstacles de Bernoulli ou entrelacs aléatoires},
	author = {Nicolas Bouchot},
	date = {2024}
}

@book{milnorMorseTheory2016,
  title = {Morse {{Theory}}},
  author = {Milnor, John},
  date = {2016-03-02},
  publisher = {Princeton University Press},
  doi = {10.1515/9781400881802},
  isbn = {978-1-4008-8180-2},
  langid = {english}
}

@article{bouchot2024scaling-homogene,
  title={Scaling limits for the random walk penalized by its range in dimension one},
  author={Bouchot, Nicolas},
  journal={ALEA},
  volume={21},
  pages={791--813},
  year={2024}
}

@online{chiariniLowerBoundsBulk2023,
  title = {Lower Bounds for Bulk Deviations for the Simple Random Walk on $\mathbb{Z}^d, d \geq 3$},
  author = {Chiarini, Alberto and Nitzschner, Maximilian},
  date = {2023-12-28},
  eprint = {2312.17074},
  eprinttype = {arXiv},
  eprintclass = {math-ph},
  doi = {10.48550/arXiv.2312.17074},
  pubstate = {prepublished},
  version = {1}
}

@article{popovSoftLocalTimes2015,
  title = {Soft Local Times and Decoupling of Random Interlacements},
  author = {Popov, Serguei and Teixeira, Augusto},
  date = {2015},
  journaltitle = {Journal of the European Mathematical Society},
  shortjournal = {J. Eur. Math. Soc.},
  volume = {17},
  number = {10},
  eprint = {1212.1605},
  eprinttype = {arXiv},
  eprintclass = {math},
  pages = {2545--2593},
  issn = {1435-9855},
  doi = {10.4171/JEMS/565},
  langid = {english}
}

@article{teixeiraInterlacementPercolationTransient2009a,
  title = {Interlacement Percolation on Transient Weighted Graphs},
  author = {Teixeira, Augusto},
  date = {2009-01-01},
  journaltitle = {Electronic Journal of Probability},
  shortjournal = {Electron. J. Probab.},
  volume = {14},
  issn = {1083-6489},
  doi = {10.1214/EJP.v14-670},
  issue = {none},
  langid = {english}
}

@book{resnickExtremeValuesRegular1987,
  title = {Extreme {{Values}}, {{Regular Variation}} and {{Point Processes}}},
  author = {Resnick, Sidney I.},
  editor = {Mikosch, Thomas V. and Resnick, Sidney I. and Robinson, Stephen M.},
  editortype = {redactor},
  date = {1987},
  series = {Springer {{Series}} in {{Operations Research}} and {{Financial Engineering}}},
  publisher = {Springer},
  location = {New York, NY},
  doi = {10.1007/978-0-387-75953-1},
  isbn = {978-0-387-75952-4 978-0-387-75953-1}
}

@article{delmotte1999parabolic,
  title={Parabolic Harnack inequality and estimates of Markov chains on graphs},
  author={Delmotte, Thierry},
  journal={Revista matem{\'a}tica iberoamericana},
  volume={15},
  number={1},
  pages={181--232},
  year={1999}
}

@misc{NIST:DLMF,
       title = {NIST Digital Library of Mathematical Functions},
         url = {https://dlmf.nist.gov/},
        note = {F.~W.~J. Olver, A.~B. {Olde Daalhuis}, D.~W. Lozier, B.~I. Schneider,
                R.~F. Boisvert, C.~W. Clark, B.~R. Miller, B.~V. Saunders,
                H.~S. Cohl, and M.~A. McClain, eds. , Release 1.2.4 of 2025-03-15},
                label = {DLMF}}
	
\end{document}